\documentclass[12pt,oneside]{amsart}
\newenvironment{aside}{\setbox0\hbox\bgroup }
  {\egroup}
\def\ShowAsides{%
  
  }
\ShowAsides
\usepackage{amssymb,amsfonts,amscd,mathrsfs}
\usepackage[left=3cm,right=2cm,top=2cm,bottom=2cm]{geometry}

\usepackage{hyperref}

\theoremstyle{plain}
\newtheorem{thm}{Theorem}[section]
\newtheorem{cor}[thm]{Corollary}
\newtheorem{prop}[thm]{Proposition}
\newtheorem{lem}[thm]{Lemma}

\theoremstyle{definition}
\newtheorem*{defn}{Definition}
\newtheorem*{rem}{Remark}

\numberwithin{equation}{section}

\renewcommand\labelenumi{\theenumi.}
\let\le\leqslant
\let\ge\geqslant
\let\sm\smallsetminus
\makeatletter
\renewcommand\sum{\DOTSB\sum@\limits}
\renewcommand\prod{\DOTSB\prod@\limits}
\makeatother
\newcommand\blank{\mathord{\hbox to 1.5ex{\hrulefill}}\,}
\let\opn\operatorname
\DeclareMathOperator\Lie{Lie}
\DeclareMathOperator\Center{C}
\DeclareMathOperator{\im}{Im}
\DeclareMathOperator{\GL}{GL}
\DeclareMathOperator{\Sp}{Sp}
\DeclareMathOperator{\SL}{SL}
\DeclareMathOperator{\SO}{SO}
\DeclareMathOperator{\Spin}{Spin}
\let\mf\mathfrak
\def\q{\mf q}
\def\p{\mf p}
\def\m{\mf m}
\def\X{\mathscr X}
\newcommand\Z{\mathbb Z}
\newcommand\N{\mathbb N}
\newcommand\Q{\mathbb Q}
\let\ph\varphi
\let\eps\varepsilon
\let\wt\widetilde
\let\ovl\overline
\newcommand\conj[2]{\vphantom{#2}^{#1}\!#2}
\title
{Normal structure of isotropic reductive groups over rings}

\author{Anastasia Stavrova}
\email{anastasia.stavrova@gmail.com}
\address{St. Petersburg Department of Steklov Mathematical Institute, nab. r. Fontanki 27, 191023 St. Petersburg, Russia}
\author{Alexei Stepanov}
\email{stepanov239@gmail.com}

\address{St. Petersburg State University, Department of Mathematics and Computer Science,
14th Line V.O. 29B, 199178 St. Petersburg, Russia}

\thanks
{The second author was supported by Russian Science Foundation grant~17-11-01261.}



\begin{document}
\begin{abstract}
The paper studies the lattice of subgroups of an isotropic reductive
group $G(R)$ over a commutative ring $R$, normalized by the elementary subgroup $E(R)$.
We prove the sandwich classification
theorem for this lattice under the assumptions that the isotropic rank of $G$ is at least 2
and the structure constants are invertible in $R$. The theorem asserts that the lattice splits into a disjoint union
of sublattices (sandwiches) $E(R,\q )\le\dots\le C(R,\q )$ parametrized by the ideals $\q$
of $R$, where $E(R,\q )$ denotes the relative elementary subgroup and $C(R,\q)$ is
the inverse image of the center under the natural homomorphism $G(R)\to G(R/\q)$.
The main ingredients of the proof are the ``level computation'' by the first author and
the generic element method developed by the second author.
\end{abstract}

\maketitle

%%==========================================================
\section{Introduction}

Let $G$ be a reductive group scheme over a unital commutative ring $R$ in the sense of~\cite{SGA}.
The famous result of J. Tits~\cite{Tits64}
establishes that if $R$ is a field, and $G$ has no normal closed connected smooth $R$-subgroups, then its
group of $R$-points $G(R)$ is very
close to being simple as an abstract group (except in a few cases where $R=\mathbb{F}_2$ or $\mathbb{F}_3$).
Namely, $G(R)$ contains a ``large'' normal subgroup $G(R)^+$ whose central
quotient is simple. If $R$ is a finite field, the corresponding simple group $G(R)^+/\Center\bigl(G(R)^+\bigr)$ is a finite
simple group of Lie type, and in fact such groups constitute the largest family  in the classification of finite simple
groups~\cite{WilsonBook}.

If $R$ is no longer a field, then every proper ideal $\q$ of $R$ determines a normal subgroup $G(R,\q)$ of $G(R)$ called
the principal congruence subgroup of level $\q$; this subgroup is the kernel
of the natural homomorphism $\rho_\q:G(R)\to G(R/\q)$. Thus $G(R)$ is no longer simple, and its lattice of normal subgroups
is at least as rich as the lattice of ideals in $R$. H. Bass~\cite{Bass64} established that,
at least if $G=\GL_r$ and $r>\max(\dim R,2)$, every normal subgroup in $G(R)$ is sandwiched
between the ``small'' elementary congruence subgroup $E(R,\q)$ and the ``large'' full congruence subgroup
$C(R,\q)$, to be defined below (the discription for $\GL_2$ is very different~\cite{CoKeSL2}). This revelation was instrumental in the proof of the congruence subgroup problem
for $\SL_n$ ($n\ge 3$) and $\Sp_{2n}$ ($n\ge 2)$ by J. Mennicke, H. Bass, M. Lazard, J. Milnor, and J.-P.
Serre~\cite{Mennicke-cong,BassLazardSerre,BassMilnorSerre}. Further work on the congruence subgroup problem for Chevalley groups~\cite{Matsumoto69},
as well as progress in algebraic $K$-theory, inspired extensions of
this normal structure result to split reductive groups over general commutative rings.
In particular, L.\,Vaserstein~\cite{VasersteinChevalley}
proved the standard normal sructure of Chevalley groups $G$ of rank $\ge 2$
over commutative rings with the only condition that the structure constants of $G$ are invertible in $R$.
The congruence subgroup problem for absolutely almost simple simply connected algebraic groups of isotropic rank $\ge 2$
over global fields was solved by
M. S. Raghunathan~\cite{Ragh76}, who obtained, as an important intermediate
step, a description of normal subgroups in arithmetic subgroups of these algebraic groups.
However, the description of normal subgroups for isotropic reductive groups over general rings has been known only in a few
distinct cases, such as orthogonal and hyperbolic unitary groups~\cite{Vas88,ZuhongSubnormal}.
The goal of the present paper is to complete this line of thought by proving that
the standard description of normal subgroups holds for all reductive groups $G$
of isotropic rank $\ge 2$, under the same necessary assumptions on $R$ as for split groups.

%In order to state our result in a precise form, we require the folloneed some definitions and notation.
%Let $A$ be a unital commutative ring. The elementary subgroup $E_l(A)$ of $\GL_l(A)$ is the subgroup generated by the
%elementary transvections $e+te_{ij}$, $1\le i\neq j\le l$, $t\in A$.
%For any reductive group scheme $G$ over $A$, satisfying a suitable isotropy condition,
%one defines an analogous elementary subgroup $E(A)$ of the group of $A$-points $G(A)$, as the subgroup
%generated by the $A$-points of unipotent radicals of parabolic subgroups of $G$; see \S~\ref{ssec:eldef} or~\cite{PS}
%for a formal definition. In particular, if $A=k$ is a field, $E(k)$ is nothing but the group $G(k)^+$ introduced
%by J. Tits~\cite{Tits64}. If $G$ is a Chevalley group with root system $\Phi$, then $E(A)=E(\Phi,A)$ is the group
%generated by the elementary root subgroups $x_{\alpha}(A)$, $\alpha\in\Phi$.
To be more precise, we say that $G$ has isotropic rank $\ge n$, if every
semisimple normal $R$-subgroup of $G$ contains an $n$-dimensional split $R$-torus
$(\mathbb{G}_{m,R})^n$. If the isotropic rank is $\ge 1$, then $G$ contains a pair of opposite
strictly proper
parabolic $R$-subgroups
$P^\pm$~\cite{SGA}, and one defines the elementary subgroup $E_P(R)$ as the subgroup of $G(R)$
generated by $U_{P^+}(R)$ and $U_{P^-}(R)$, where $U_{P^\pm}$ denotes the unipotent radical  of $P^\pm$.
If, moreover, the isotropic rank of $G$ is $\ge 2$, the main result of~\cite{PetStavIso} implies that $E(R)=E_P(R)$
is independent of the choice of $P^\pm$ and is normal in $G(R)$ (see \S~\ref{sec:elem} for the details).
%Under these assumptions,
If $G=\GL_n$, $n\ge3$, then $E(R)$ is the usual elementary subgroup
generated by the elementary transvections $e+te_{ij}$, $1\le i\neq j\le n$, $t\in R$
(here $e$ denotes the identity matrix whereas $e_{ij}$ is the matrix with 1 at position $(i,j)$
and zeros elsewhere).
If $R$ is a field, then $E(R)=G(R)^+$ is the above-mentioned group of J.\,Tits.
%If $G$ is a Chevalley group with root system $\Phi$, then $E(A)=E(\Phi,A)$ is the group
%generated by the elementary root subgroups $x_{\alpha}(A)$, $\alpha\in\Phi$.
%Let us mention that the functor $K_1^G(-)=G(-)/E(-)$ on the category of commutative $R$-algebras
%is called the non-stable, or unstable, $K_1$-functor associated to $G$, or
%the Whitehead group of $G$. The name is
%due to the fact that the functors $K_1^G$ are similar in many aspects to the $K_1$-functor of algebraic $K$-theory,
%and its study goes back to Bass' founding paper~\cite{Bass}, where $K_1^{\GL_l}$ was considered
%in relation to stabilization problems. In the field case the study of $K_1^G$
%is subject of the Kneser--Tits problem, see~\cite{Gil} and references therein.

Assume that the isotropic rank of $G$ is $\ge 2$. For any ideal $\q$ of $R$, let $U_{P^\pm}(\q)$
be the kernel of the homomorphism $\rho_\q$ restricted to $U_{P^\pm}(R)$. Denote by
$E_P(R,\q)$ the normal closure of the subgroup generated by $U_{P^+}(\q)$ and $U_{P^-}(\q)$ in $E(R)$.
The full congruence subgroup $C(R,\q)$ is the full preimage of the center of $G(R/\q)$ under $\rho_\q$.

For any maximal ideal $\m$ of $R$, denote by $\overline{R/\m}$ the algebraic closure of the field
$R/\m$. The group $G_{\overline{R/\m}}$ is a reductive algebraic group in the usual sense~\cite{BorelBook}, and thus
has a root system $\Phi$ in the sense of Bourbaki~\cite{Bourbaki4-6}.
Note that $G_{\overline{R/\m}}$ has no closed connected normal smooth $R$-subgroups if and only if $\Phi$ is irreducible.
The structure constants of $\Phi$ are, by definition, the integers $\pm 1$,
together with $\pm 2$, if $\Phi$
is of type $B_n,C_n,F_4$, or $\pm 2,\pm 3$, if $\Phi$ is of type $G_2$.

The main result of the present paper is the following theorem.

\begin{thm}\label{thm:main}
Let $G$ be a reductive group scheme over a ring $R$ such that its isotropic rank is $\ge 2$.
Suppose that for every maximal ideal $\m$ of $R$ the root system of
$G_{\overline{R/\m}}$ is irreducible, and its structure constants are invertible in $R$.
Let $P$ be a proper parabolic subgroup of $G$.
Then
\def\labelenumi{\normalfont{(\roman{enumi})}}
\begin{enumerate}
\item
For any ideal $\q$ of $R$ one has $E_P(R,\q)=[G(R,\q),E(R)]$. In particular,
$E_P(R,\q)=E(R,\q)$ is independent of the choice of a parabolic $R$-subgroup $P$.\\
\item
For any subgroup $H\le G(R)$
normalized by $E(R)$, there exists a unique ideal $\q$ in $R$ such that
$$E(R,\mf q)\le H\le C(R,\q).$$
\end{enumerate}
\end{thm}

Since groups of points of reductive group schemes include, in particular, the linear matrix groups
$\GL_n(R)$, $\SL_n(R)$, $\Sp_{2n}(R)$, as well as spinor and special orthogonal groups $\Spin(q)(R)$, $\SO(q)(R)$, where
$q$ is a non-degenerate quadratic $R$-form, the study of their normal structure has a very long history.
Thus, it would take a separate survey paper to describe it in full, and below we only list
a few milestone results.

\begin{itemize}
\item
Simplicity of the groups $\opn{PSL}_n(F)$ was proved by Camille Jordan around 1870 (for prime fields)
and in early 1900 by Leonard Dickson (for all finite fields).
\item
A similar result for the $\opn{GL}_n$ over a skew-field was obtained by J. Dieudonne in 1943. In~\cite{DieuClassiques}
he established the simplicity of split groups over arbitrary fields. The first uniform proof of their simplicity was given by
C.\,Chevalley~\cite{ChevalleySimpGroups}.
\item
The problem over rings distinct from fields was first approached by J. Brenner~\cite{Brenner38,Brenner44,Brenner60} and
later by W. Klingenberg~\cite{KlingenbergGL,KlingenbergSO,KlingenbergSp}. They studied split groups of classical types
over quotients of $\mathbb{Z}$ and general local rings respectively.

\item As mentioned above, J. Tits established the simplicity for isotropic reductive groups over fields in~\cite{Tits64}.
\item For the general linear group over an arbitrary ring,
the normal structure theorem was proved by H. Bass~\cite{Bass64}
under the stable range condition. The result was generalized to other quasi-split classical
groups by H. Bass himself~\cite{BassUnitary} and A. Bak~\cite{BakClassical}.

\item H. Bass, M. Lazard, J.-P. Serre, and J. Milnor~\cite{BassLazardSerre,BassMilnorSerre}
elucidated the normal structure of $\SL_n$ and $\Sp_{2n}$ over
rings of integers of algebraic number fields, and stated the congruence subgroup problem.
It was solved for $\SL_2$ by J.-P. Serre~\cite{Serre-cong} and for all split groups by H. Matsumoto~\cite{Matsumoto69}.

\item The stable range condition for $\opn{GL}_n$ was removed by
J.\,Wilson~\cite{WilsonNormal} ($n\ge4$) and I.\,Golubchik~\cite{GolubchikNormal} ($n\ge3$).
Using the ideas of H.\,Bass and Suslin's theorem on the normality of the elementary group~\cite{SuslinSerreConj},
Z.\,Borewich and N.\,Vavilov~\cite{BV84} gave a simpler prove of the Wilson--Golubchik theorem.

\item In 1974 E.\,Abe and K.\,Suzuki proved the normal structure theorem for all Chevalley groups over local
rings~\cite{AbeSuzuki}.

\item In 1976 M.S. Raghunathan established the congruence subgroup problem for groups of isotropic rank $\ge 2$ over
a global field~\cite{Ragh76}. In 1986 he improved this result, weakening the isotropy conditions~\cite{Ragh86}.

\item In 1979 G. Margulis~\cite{Margulis79} proved his celebrated theorem on lattices in isotropic groups over local
fields. His work was based on earlier work of V. P. Platonov~\cite{Platonov69}.

\item In 1980 R. Bix established the normal structure theorem for isotropic groups of type $\conj1E_{6,2}^{28}$ over
local rings~\cite{Bix80}.

\item
L.\,Vaserstein in~\cite{VasersteinGLn} combined the results of H.\,Bass, J.\,Wilson and I.\,Golubchik
proving the standard normal structure of the $\opn{GL}_n$ under a local stable rank condition.

\item
After the result of G.\,Taddei on normality of the elementary group in a Chevalley group
L.\,Vaserstein~\cite{VasersteinChevalley} proved the standard normal sructure of Chevalley groups of rank $\ge 2$
over commutative rings
provided that the structure constants were invertible.
The latter condition was removed by E.\,Abe in~\cite{AbeNormal}.

\item In 1988 L. Vaserstein proved the normal structure theorem for isotropic orthogonal groups over commutative
rings~\cite{Vas88}.

\item The result of E.\,Abe has three exceptions where the elementary group is not perfect:
types $C_2$ and $G_2$ if the gound ring has a residue field of 2 elements, and type $A_1$.
All the exceptions were considered by D.\,Costa and G.\,Keller in a series of papers~\cite{CoKeSL2,CoKeSp,CoKeG2}.
Of course, for type $A_1$ the ground ring must be low-dimensional with many units, as the
 description of the normal structure of $\opn{SL}_2(\mathbb Z)$ seems to be an unrealistic problem.

\item In 2010 Zuhong Zhang~\cite{ZuhongSubnormal} established the normal structure theorem for even hyperbolic
unitary groups over a commutative form ring with invertible $2$. The latter assumption was removed by Hong You~\cite{HongYou12}.
A shorter proof was obtained recently by R.~Preusser~\cite{Preusser18}.
\end{itemize}

The main ingredients of the proof of Theorem~\ref{thm:main} is a level computation made by the first author
in~\cite{StavCongruence} and the generic element method developed by the second author in~\cite{StepUniloc}.
Our proof basically follows the plan developed by the second author in~\cite{StepDecomp}.

\section{Basic notation and conventions}\label{sec:notation}
Let $x,y,z$ be elements of an abstract group $H$. Denote by $\conj yx=yxy^{-1}$ the left
conjugate to $x$ by $y$. Sometimes we use also the right conjugate $x^y=y^{-1}xy$ to
$x$ by $y$. The commutator $xyx^{-1}y^{-1}$ is denoted by $[x,y]$.
In the sequel we frequently use the following commutator identity, which can be easily verified
by a straightforward calculation.
%% Let $x,y,z$ be elements of an abstract group $H$. Then
%
\begin{equation}\label{xyzz-1}
[zy,x]^z=\conj yx\cdot(x^{-1})^z=[y,x]\cdot[x,z^{-1}]
\end{equation}

Let $S$ be a subset of $H$. By $\langle S\rangle$ we denote the subgroup
spanned by $S$. For  subsets $X$ and $Y$ of $H$
$[X,Y]$ stands for the subgroup of $H$ generated by all the commutators
$[x,y]$, $x\in X$, $y\in Y$. Note that if $Y$ is a subgroup, then $[X,Y]$ is normalized by $Y$
(apply~\eqref{xyzz-1} to the case, where $x\in X$ and $z$ and $zy$ are arbitrary elements of $Y$).
If both $X$ and $Y$ are subgroups, then we denote by $X^Y$ the subgroup of $H$
generated by $x^y$ for all $x\in X$ and $y\in Y$. In other words, $X^Y$ is the smallest subgroup
containing $X$ and normalized by $Y$.
The centralizer of a subgroup $H'\le H$ is denoted by $\Center_H(H')$ and
we write $\Center(H)=\Center_H(H)$ for the center of $H$.

All rings and algebras are assumed to be commutative and to contain a unit.
All homomorphisms preserve unit elements.
The multiplicative group of a ring $R$ is denoted by $R^\times$.
As usual, $\opn{Spec} R$ denotes the prime spectrum of $R$.
For $\p\in\opn{Spec} R$ denote by $R_{\p}=(R\sm\p)^{-1}R$
the localization of $R$ at $\mf p$ and by $k(\p)$ the residue
field $R_\p/\p R_\p$. The algebraic closure of $k(\p)$ is denoted by $\overline{k(\p)}$.

Let $s\in R$. The principal localization at the element $s$
(i.\,e. the localization at the multiplicative subset generated by $s$)
is denoted by $R_s$. The localization homomorphism is denoted by $\lambda_{\mf p}$
or $\lambda_s$ respectively.

Throughout the paper $K$ is a ring, $G$ stands for a reductive group scheme over $K$ in the sense of~\cite{SGA},
and $R$ denotes a $K$-algebra, unless explicitly stated otherwise. Sometimes we replace $G$ by $G_R$
and consider it as a group scheme over $R$. It makes no harm as we deal with
the group of points $G(R)=G_R(R)$.

For any ideal $\q$ of $R$ we denote
by $\rho_\q:R\to R/\q$ the reduction homomorphism, and, by abuse of notation, the induced homomorphism $G(R)\to G(R/\q)$.
The principal congruence subgroup $G(R,\mf q)$ is the kernel of $\rho_{\mf q}:G(R)\to G(R/\mf q)$,
whereas the full congruence subgroup $C(R,\q)$ is the inverse image of the center $\Center\bigl(G(R/\q)\bigr)$ of
$G(R/\mf q)$ under this homomorphism.

An ideal $\q$ of $R$ is called \emph{splitting} if $R/\q$ embeds into $R$ such that the composition
$R/\q\rightarrowtail R\twoheadrightarrow R/\q$ is identity. In other words, $\q$ is a splitting ideal
if $R=R'\oplus\q$ as additive groups, where $R'$ is a subring of $R$.

%%==========================================================
\section{Generic element techniques}
Let $G$ be an affine smooth finitely presented group scheme over $K$.
%finitely presented may be excessive, since smooth is always locally finitely presented
Denote by $A=K[G]$ the affine algebra of the scheme $G$.
By the definition of an affine scheme, an element $h\in G(R)$ can be identified with a
homomorphism $h:A\to R$.
We always do this identification, i.\,e. we always view elements
of the group of points $G(R)$ of the scheme $G$ over a $K$-algebra $R$ as homomorphisms
from $A$ to $R$.\footnote
{Thus, for $f\in A$ we write $h(f)$ instead of $f(h)$, as one writes
considering the affine algebra as a set of functions.}
Denote by $g\in G(A)$ the generic element of the scheme $G$, i.\,e.
the identity map $\mathrm{id}_A:A\to A$. An element $h\in G(R)$ induces the homomorphism
$G(h):G(A)\to G(R)$ by the rule $G(h)(a)=h\circ a$ for all $a\in G(A)$.
It follows that the image of $g$ under the action of $G(h)$ is equal to~$h$.

For a ring homomorphism $\ph:R\to R'$ we usually  denote the induced group homomorphism
$G(\ph):G(R)\to G(R')$ again by $\ph$.
This cannot lead to a confusion as one always can determine the meaning of $\ph$
by the argument type of this homomorphism.
In view of this agreement we have $h(g)=h\circ\operatorname{id}_A=h$.
If $R'$ is an $R$-algebra, then sometimes we identify elements of $G(R)$ with their canonical
images in $G(R')$.

Recall that the fundamental ideal $I$ of $A$ is the kernel of the counit map $e_K:A\to K$, where
$e_K\in G(K)$ is the identity element of this abstract group.
\emph{The notation $A$, $I$, and $g$ introduced above is kept till the end of the present section.}

The following characterization
of the principal congruence subgroup was observed in~\cite{StepUniloc}, see the paragraph before Lemma~2.1.

\begin{lem}\label{lem:GRq}
An element $h$ of $G(R)$ belongs to the principal congruence subgroup $G(R,\q)$ if and only
if $h(I)\subseteq\q$.
\end{lem}

\begin{proof}
Denote by $\rho_\q:G(R)\to G(R/\q)$ the natural homomorphism.
It is easy to see that $h\in G(R,\q)$ iff the following diagram commutes.
$$
\begin{CD}
A  @>h>>   R   \\
@V{e_K}VV  @VV{\rho_\q}V\\
K  @>>>    R/\q\\
\end{CD}
$$
And the latter is obviously equivalent to saying that $h(I)$ vanishes modulo $\q$,
i.\,e. $h(I)\subseteq\q$.
\end{proof}

Since $A$ is a finitely presented $K$-algebra, it is a quotient of a polynomial ring in
finitely many variables by a finitely generated ideal. Let $K'$ be a $\Z$-subalgebra of $K$,
generated by all the coeffients of polynomials that generate this ideal. Then $A\cong A'\otimes_{K'}K$,
were $A'$ is a finitely generated algebra over a Noetherian ring $K'$. Thus, there exists
an affine smooth finitely presented group scheme $\wt G$ over a Noetherian ring $K'$ such that $G=\wt G_K$.
In the present paper we prove results about the abstract group $G(R)$ for a $K$-algebra $R$.
Therefore, without loss of generality we may assume that $K=K'$, and hence $A$ is a
Noetherian ring.

An advantage of the Noetherian property of a ring in this context is the following
property~\cite[Lemma~4.10]{BakNonabelian}.

\begin{lem}\label{lem:Bak}
Let $R$ be a Noetherian ring and $s\in R$. There exists $m\in\N$ such that the restriction
of the localization homomorphism $\lambda_s:R\to R_s$ to the ideal $s^mR$ is injective.
\end{lem}

The next lemma is a version of clearing denominators.

\begin{lem}\label{lem:ClearDenom}
Let $G$ be an affine smooth finitely presented group scheme over a ring $K$, \ $R$ a Noetherian $K$-algebra,
$m\in\N$, and $s\in R$.
Suppose that the natural map $s^mR\to R_s$ is injective.
Given $a\in G(R_s)$ there exists $k\in\N$ such that $[a,\lambda_s(b)]\in\lambda_s\bigl(G(R,s^mR)\bigr)$
for all $b\in G(R,s^kR)$.
\end{lem}

\begin{proof}
Since $A$ is finitely presented, there exists a finite
set $J\subseteq A$ that generates $A$ as a $K$-algebra.
Since $A=K\oplus I$ as a $K$-module, we may assume that $J\subseteq I$.
Obviously, under this condition $J$ spans $I$ as an ideal.

Identifying elements of $G(A)$ and $G(R_s)$ with their canonical images in $G(A\otimes_K R_s)$,
consider the commutator $c=[a,\lambda_s(g)]\in G(A\otimes_K R_s,I\otimes_KR_s)$.
By Lemma~\ref{lem:GRq} the finite set $c(J)$ is contained in $I\otimes_K R_s$. Clearly, there exists $l\in\N$
such that each element of this set can be written as $r/s^l$ for some $r\in I\otimes_K R$.
Take $k=l+m$. Take $b\in G(R,s^kR)$ and consider the composition
$d=\mathrm{mult}\circ(b\otimes\mathrm{id})\circ c=[a,\lambda_s(b)]$, where
$\mathrm{mult}:R\otimes_KR_s\to R_s$ denotes the multiplication homomorphism. By Lemma~\ref{lem:GRq}
$b(I)\subseteq s^kR$, hence the set $d(J)$ consists of the elements of the form
$\mathrm{mult}\circ(b\otimes\mathrm{id})(r/s^l)=s^kt/s^l=s^mt$ for some $t\in R$.
Since the restriction of the localization homomorphism $\lambda_s$ to $s^mR$ is injective
and $J$ generates the $K$-algebra $A$, the homomorphism $d$ factors through $d':A\to R$,
which means that $d=\lambda_s(d')$. Moreover, the natural choice of $d'$ provides that
$d'(J)\subseteq s^mR$. Since $J$ generates the fundamental ideal $I$, we have
$d'(I)\subseteq s^mR$ and by Lemma~\ref{lem:GRq} $d'\in G(R,s^mR)$.
\end{proof}

%%==========================================================
\section{Elementary subgroup of an isotropic reductive group}\label{sec:elem}
Let $P$ be a parabolic subgroup of $G$ in the sense of~\cite{SGA}.
Since the base $\opn{Spec}R$ is affine, the group $P$ has a Levi subgroup $L_P$~\cite[Exp.~XXVI Cor.~2.3]{SGA}.
There is a unique parabolic subgroup $P^-$ in $G$ which is opposite to $P$ with respect to $L_P$,
that is $P^-\cap P=L_P$, cf.~\cite[Exp. XXVI Th. 4.3.2]{SGA}.  We denote by $U_P$ the unipotent
radical of $P$.

Note that if $L'_P$ is another Levi subgroup of $P$,
then $L'_P$ and $L_P$ are conjugate by an element $u\in U_P(R)$~\cite[Exp. XXVI Cor. 1.8]{SGA}.
Since our proofs in the present paper do not depend on a particular choice of $L_P$ or $P^-$,
we do not pay attention to this choice.

\begin{defn}
The \emph{elementary subgroup $E_P(R)$ corresponding to $P$} is the subgroup of $G(R)$
generated as an abstract group by $U_P(R)$ and $U_{P^-}(R)$.
\end{defn}

\begin{defn}
A parabolic subgroup $P$ in $G$ is called
\emph{strictly proper}, if it intersects properly every normal semisimple subgroup of $G$.
\end{defn}

The following theorem is the main result of~\cite{PetStavIso}.

\begin{thm}[{\cite[Theorem 1]{PetStavIso}}]\label{thm:EE}
Assume that for every maximal ideal $\m$ of $R$ every normal semisimple
subgroup of $G_{R_\m}$ contains $(\mathbb G_{m,R_\m})^2$. Then the subgroup $E_P(R)$ of $G(R)=G_R(R)$
is the same for any strictly proper parabolic $R$-subgroup $P$ of $G_R$.
In particular, $E_P(R)$ is normal in $G(R)$.
\end{thm}

\begin{defn}
Under the assumptions of Theorem~$\ref{thm:EE}$ we call $E_P(R)$ \emph{the elementary subgroup} of $G(R)$ and denote it by $E(R)$.
\end{defn}

We also use the following theorems, which are the main results of~\cite{KulikStav} and~\cite{LuzStav}.
We denote by $\Center(G)$ the group scheme center of $G$ in the sense of~\cite{SGA}, see also discussion
around Proposition~6.7 in~\cite{MilneAGS}. The very definition of $\Center(G)$ implies that
$\Center(G)(R)\le \Center\bigl(G(R)\bigr)$.

\begin{thm}[{\cite[Theorem 1]{KulikStav}}]\label{thm:E-cent}
Under the hypothesis of Theorem~$\ref{thm:EE}$, one has
$$
\Center_{G(R)}\bigl(E(R)\bigr)=\Center(G)(R)=\Center\bigl(G(R)\bigr).
$$
\end{thm}

\begin{thm}[{\cite[Theorem 1]{LuzStav}}]\label{thm:perfect}
Under the hypothesis of Theorem~$\ref{thm:EE}$,
assume, moreover, that for any maximal ideal $\m$ of $R$, one has
$k(\m)\neq\mathbb{F}_2$ whenever the root system of $G_{\overline{k(\m)}}$
contains an irreducible component of type $B_2=C_2$ or $G_2$.
Then $E(R) = [E(R),E(R)]$.
\end{thm}

The following statement will be used twice to get rid of the center.

\begin{lem}\label{lem:HallWitt}
Let $H$ be a subgroup of $G(R)$, normalized by $E(R)$. Under the assumptions of Theorem~$\ref{thm:perfect}$
$\bigl[[H,E(R)],E(R)\bigr]=[H,E(R)]$.
In particular, if $[H,E(R)]\le\Center\bigl(G(R)\bigr)$, then $H\le\Center\bigl(G(R)\bigr)$.
\end{lem}

\begin{proof}
By Theorem~\ref{thm:perfect} the group $E(R)$ is perfect.
Using the Hall--Witt identity we get
$$
\bigl[E(R),H\bigr]=\bigl[[E(R),E(R)],H\bigr]\le\bigl[[H,E(R)],E(R)\bigr].
$$
The inverse inclusion is obvious. The second assertion follows immediately from the first one and
Theorem~\ref{thm:E-cent}.
\end{proof}

%%==========================================================
\section{Root systems corresponding to parabolic subgroups}
%% (from~\cite{StavMulti})
Let $S=(\mathbb G_{m,R})^N=\opn{Spec}(R[x_1^{\pm 1},\ldots,x_N^{\pm 1}])$
be a split $N$-dimensional torus over $R$. Recall that the character group
$X^*(S)=\opn{Hom}_R(S,\mathbb G_{m,R})$ of $S$ is canonically isomorphic to $\Z^N$.
If $S$ acts $R$-linearly on an $R$-module $V$, this module has a natural $\Z^N$-grading
$$
V=\bigoplus_{\lambda\in X^*(S)}V_\lambda,
$$
where
$$
V_\lambda=\{v\in V\mid s\cdot v=\lambda(s)v\text{ for any }s\in S(R)\}.
$$
Conversely, any $\Z^N$-graded $R$-module $V$ can be provided with an $S$-action by the same rule.

Assume that $S$ acts on $G$ by $R$-group automorphisms.
The associated Lie algebra functor $\Lie(G)$ then acquires
a $\Z^N$-grading compatible with the Lie algebra structure,
$$
\Lie(G)=\bigoplus_{\lambda\in X^*(S)}\Lie(G)_\lambda.
$$

We will use the following version of~\cite[Exp. XXVI Prop. 6.1]{SGA}.

\begin{lem}\label{lem:T-P}
Let $L=\Center_G(S)$ be the subscheme of $G$ fixed by $S$. Let
$\Psi\subseteq X^*(S)$ be an $R$-subsheaf of sets closed under addition of characters.
Then there exists a unique smooth connected closed subgroup $U_\Psi$ of $G$
normalized by $L$ and satisfying
\begin{equation}\label{eq:LieUPsi}
\Lie(U_\Psi)=\bigoplus_{\lambda\in\Psi}\Lie(G)_\lambda.
\end{equation}
Moreover,
\begin{enumerate}
\item
if $0\in\Psi$, then $U_\Psi$ contains $L$;
\item
if $\Psi=\{0\}$, then $U_\Psi=L$;
\item
if $\Psi=-\Psi$, then $U_\Psi$ is reductive;
\item
if $\Psi\cup(-\Psi)=X^*(S)$, then $U_\Psi$ and $U_{-\Psi}$ are two opposite parabolic subgroups of $G$
with the common Levi subgroup $U_{\Psi\cap(-\Psi)}$;
\item
If $0\notin\Psi$, then $U_\Psi$ is unipotent.
\end{enumerate}
\end{lem}

\begin{proof}
The statement immediately follows by faithfully flat descent from the standard facts about the subgroups of
split reductive groups proved in~\cite[Exp. XXII]{SGA}; see the proof of~\cite[Exp. XXVI Prop. 6.1]{SGA}.
\end{proof}

\begin{defn}
The sheaf of sets
$$
\Phi(S,G)=\{\lambda\in X^*(S)\sm\{0\}\ |\ \Lie(G)_\lambda\neq 0\}
$$
is called the \emph{system of relative roots of $G$ with respect to $S$}.
\end{defn}

\begin{rem}%\label{rem:order}
Choosing a linear ordering on the $\Q$-space $\Q\otimes_{\Z} X^*(S)\cong\Q^n$, one defines the subsets
of positive and negative relative roots $\Phi(S,G)^+$ and $\Phi(S,G)^-$, so that $\Phi(S,G)$ is a disjoint
union of $\Phi(S,G)^+$, $\Phi(S,G)^-$, and $\{0\}$. By Lemma~$\ref{lem:T-P}$ the closed subgroups
$$
U_{\Phi(S,G)^+\cup\{0\}}=P,\qquad U_{\Phi(S,G)^-\cup\{0\}}=P^-
$$
are two opposite parabolic subgroups of $G$ with the common Levi subgroup $\Center_G(S)$.
Thus, if a reductive group $G$ over $R$ admits a non-trivial action of a split torus,
then it has a proper parabolic subgroup. The converse is true Zariski-locally,
see Lemma~$\ref{lem:relroots}$ below.
\end{rem}

Let $P$ be a parabolic subgroup scheme of $G$ over $R$, and let $L$ be a Levi subgroup of $P$.
By~\cite[Exp. XXII, Prop. 2.8]{SGA} the root system $\Phi$ of $G_{\overline{k(\p)}}$, $\p\in\opn{Spec}R$,
is locally constant in the Zariski topology on $\opn{Spec}R$. The type of the root system of
$L_{\overline{k(\p)}}$ is determined by a Dynkin subdiagram
of the Dynkin diagram of $\Phi$, which is also constant Zariski-locally on $\opn{Spec}R$
by~\cite[Exp. XXVI, Lemme 1.14 and Prop. 1.15]{SGA}. In particular, if $\opn{Spec}R$ is connected,
all these data are constant on $\opn{Spec}R$.

\begin{defn}
Assume that the root system $\Phi$ of $G_{\overline{k(\p)}}$ has the same type for all $\p\in\opn{Spec}R$.
Then we call $\Phi$ \emph{the absolute root system of $G$.}
\end{defn}

\begin{lem}[{\cite[Lemma 3.6]{StavMulti}}]\label{lem:relroots}
Assume that $R$ is connected. Let $\bar L$ be the image of $L$ under the natural
homomorphism $G\to G^{\mathrm{ad}}\subseteq\opn{Aut}(G)$.
Let $D$ be the Dynkin diagram of the absolute root system $\Phi$ of $G$.
We identify $D$ with a set of simple roots of $\Phi$ such that
$P_{\overline{k(\p)}}$ is a standard positive parabolic
subgroup with respect to $D$. Let $J\subseteq D$
be the set of simple roots such that $D\sm J$ is the subdiagram of $D$ corresponing to $L_{\overline{k(\p)}}$.
Then there are a unique maximal split subtorus
$S\subseteq\Center(\bar L)$ and a subgroup $\Gamma\le\opn{Aut}(D)$ such that $J$ is invariant under $\Gamma$
and
for any $\p\in\opn{Spec} R$ and any split maximal torus $T\subseteq\bar L_{\overline{k(\p)}}$
the kernel of the natural surjection
\begin{equation}\label{eq:T-S}
X^*(T)\cong\Z\Phi\xrightarrow{\ \pi\ } X^*(S_{\overline{k(\p)}})\cong \Z\Phi(S,G)
\end{equation}
is generated by all roots $\alpha\in D\sm J$,
and by all differences $\alpha-\sigma(\alpha)$, $\alpha\in J$, $\sigma\in\Gamma$.
\end{lem}

In~\cite{PetStavIso}, we introduced a system of relative roots $\Phi_P$ with respect to a parabolic
subgroup $P$. This system $\Phi_P$ was defined
independently over each member $\opn{Spec}R=\opn{Spec}R_i$
of a suitable finite disjoint Zariski covering
$$
\opn{Spec}R=\coprod\limits_{i=1}^m\opn{Spec} R_i,
$$
such that over each $R_i$, $1\le i\le m$, the root system $\Phi$ and the Dynkin diagram $D$ of $G$ is constant.
Namely, we considered the formal projection
$$
\pi_{J,\Gamma}\colon\Z\Phi
\longrightarrow \Z\Phi/\langle D\sm J;\ \alpha-\sigma(\alpha)\mid \alpha\in J,\ \sigma\in\Gamma\rangle,
$$
and set $\Phi_P=\Phi_{J,\Gamma}=\pi_{J,\Gamma}(\Phi)\sm\{0\}$. The last claim of Lemma~\ref{lem:relroots}
allows to identify $\Phi_{J,\Gamma}$ and $\Phi(S,G)$ whenever $\opn{Spec} R$ is connected.

\begin{defn}
In the setting of Lemma~$\ref{lem:relroots}$ we call $\Phi(S,G)\cong\Phi_{J,\Gamma}$ a \emph{system
of relative roots with respect to the parabolic subgroup $P$ over $R$} and denote it by $\Phi_P$.
\end{defn}

If $R$ is a field or a local ring, and $P$ is a minimal parabolic subgroup of $G$,
then $\Phi_P$ is nothing but the relative root system of $G$ with respect to a maximal split subtorus
in the sense of~\cite{BorelTits} or, respectively,~\cite[Exp. XXVI \S 7]{SGA}.

We have also defined in~\cite{PetStavIso} irreducible components of systems of relative roots,
the subsets of positive and negative
relative roots, simple relative roots, and the height of a root. These definitions are immediate analogs
of the ones for usual abstract root systems, so we do not reproduce them here.
The height of a root $\alpha\in\Phi$ is deonted by $\opn{ht}\alpha$.

We will need later the following two lemmas on relative roots.

\begin{lem}\label{lem:adj-simple-roots}
Let $\Phi$ be a root system, and let $\Phi_{J,\Gamma}=\pi(\Phi)\sm\{0\}$ be a relative root system with the canonical
projection $\pi:\Z\Phi\to\Z\Phi_{J,\Gamma}$.
Let $\alpha,\beta\in\Phi_{J,\Gamma}$ be two simple relative roots such that $\alpha+\beta\in\Phi_{J,\Gamma}$.
Then for any $j\ge 1$ such that $j\beta\in\Phi_{J,\Gamma}$ one has $\alpha+j\beta\in\Phi_{J,\Gamma}$.
\end{lem}

\begin{proof}
The inclusion $\alpha+\beta\in\Phi_{J,\Gamma}$ is equivalent to the
existence of two simple roots in $\pi^{-1}(\beta)$ and $\pi^{-1}(\alpha)$ respectively that
are connected in the Dynkin diagram by a chain (possibly, empty)
of simple roots lying in $\pi^{-1}(0)$. Then for any element $\mu\in\pi^{-1}(j\beta)$ the set $S_\mu\subseteq\Phi$ of simple
roots  occurring in its decomposition
is also connected to a simple root in $\pi^{-1}(\alpha)$ by a chain of simple roots lying in $\pi^{-1}(0)$,
which allows to find a root $\nu\in\pi^{-1}(\alpha)$ such that $S_\mu$ and $S_\nu$
are disjoint but adjacent subsets of the Dynkin diagram of $\Phi$, and hence $\mu+\nu\in\Phi$. Then
$\pi(\mu+\nu)=\alpha+j\beta\in\Phi_{J,\Gamma}$.
\end{proof}

\begin{lem}\label{lem:parab-centr-root}
Let $\Phi$ be a root system with the scalar product $(\blank,\blank)$ and a system of simple roots
$\Pi$, and let $\Phi_{J,\Gamma}=\pi(\Phi)\sm\{0\}$ be a relative root system with the canonical
projection $\pi:\Z\Phi\to\Z\Phi_{J,\Gamma}$. Let $\beta\in\pi(\Pi)\cap \Phi_{J,\Gamma}$ be a simple relative root.
Then there is a proper parabolic subset $\Sigma$ of $\Phi_{J,\Gamma}$
that contains all $\alpha\in\Phi_{J,\Gamma}$ such that
$\alpha+\beta\notin\Phi_{J,\Gamma}$. If $\Phi$ is simply laced, then
\begin{equation}\label{eq:Sigma(beta)}
\begin{aligned}
\Sigma&=\{\alpha\in\Phi_{J,\Gamma}\mid (a,\sum_{b\in\pi^{-1}(\beta)} b)\ge 0\ \mbox{for all}\ a\in\pi^{-1}(\alpha)\}\\
      &=\{\alpha\in\Phi_{J,\Gamma}\mid (a,\sum_{b\in\pi^{-1}(\beta)} b)\ge 0\ \mbox{for some}\ a\in\pi^{-1}(\alpha)\}.
\end{aligned}
\end{equation}
\end{lem}

\begin{proof}
If $\Phi$ is not simply laced, then there is a simply laced root system $\Phi'$ with a system of simple roots $\Pi'$
and a projection $\pi':\Z\Phi'\to\Z\Phi$
such that $\Phi=\Phi'_{\Pi',\Gamma'}$
for a suitable group of automorphisms $\Gamma'$ of the Dynkin diagram of $\Phi'$.
 Then
$\Phi_{J,\Gamma}=\Phi'_{\pi'^{-1}(J),\Gamma''}$, where $\Gamma''$ is a group of automorphisms of the Dynkin diagram
of $\Phi$ generated by $\Gamma'$ and the natural lifting of $\Gamma$. This means that we can assume that
$\Phi$ is simply laced from the start.

For any $\alpha\in\Phi_{J,\Gamma}$ and any $\mu\in\pi^{-1}(\alpha)$ by~\cite[Lemma 3]{PetStavIso} the set $\pi^{-1}(\alpha)$
is a union of roots $\mu'\in\Phi$ such that $\sigma(\mu_J)=\mu'_J$ for some $\sigma\in\Gamma$, where
$\mu_J$ denotes the shape of $\mu$ with respect to $J$ in the sense of~\cite{PetStavIso} or~\cite{AzadBarrySeitz}.
By~\cite[Lemma 1]{AzadBarrySeitz}
the Weyl group of $\Pi\sm J$ acts transitively on the set of roots of a fixed shape and length, hence
this Weyl group and $\Gamma$ together act transitively on $\pi^{-1}(\alpha)$. Since all elements of these groups
are bijections that preserve scalar products and the set $\pi^{-1}(\beta)$, we conclude that
$(\mu,\sum_{\nu\in\pi^{-1}(\beta)}\nu)$ is the same for all $\mu\in\pi^{-1}(\alpha)$.
Thus, the set $\Sigma$ of~\eqref{eq:Sigma(beta)}
is well-defined, and it is clear that $\Sigma\cup(-\Sigma)=\Phi_{J,\Gamma}$.

Assume that
$\alpha_1,\alpha_2\in\Sigma$ and $\alpha_1+\alpha_2\in\Phi_{J,\Gamma}$. By~\cite[Lemma 4]{PetStavIso} for every
$\mu\in\pi^{-1}(\alpha_1+\alpha_2)$
there are $\mu_1\in\pi^{-1}(\alpha_1)$ and $\mu_2\in\pi^{-1}(\alpha_2)$ such that $\mu=\mu_1+\mu_2$.
This shows that $\Sigma$ is additively closed. By the same argument, if $\alpha+\beta\notin\Phi_{J,\Gamma}$, then
$(\mu,\nu)\ge 0$ for all $\mu\in\pi^{-1}(\alpha)$ and $\nu\in\pi^{-1}(\beta)$, and hence $\alpha\in\Sigma$.

It remains to show that $\Sigma$ is a proper subset of $\Phi_{J,\Gamma}$. Let $\alpha\in\pi(\Pi)\sm\{0\}$ be
a simple relative root such that $\alpha+\beta\in\Phi_{J,\Gamma}$. Since $\alpha-\beta\notin\Phi_{J,\Gamma}$,
we conclude that $(\mu,\nu)\le 0$ for all $\mu\in\pi^{-1}(\alpha)$ and $\nu\in\pi^{-1}(\beta)$.
On the other hand, since $\alpha+\beta\in\Phi_{J,\Gamma}$,
there are two simple roots  $\mu\in\pi^{-1}(\alpha)\cap\Pi$ and $\nu\in\pi^{-1}(\beta)\cap\Pi$ which are connected
on the Dynkin diagram by a chain of roots in $\Pi\sm J$. Then the scalar product of $\mu$ and $(\nu+$all roots in this chain$)$
is negative. Hence $\alpha\notin\Sigma$.
\end{proof}

%%==========================================================
\section{Relative root subschemes}
For any finitely generated projective $R$-module $V$, we denote by $W(V)$ the natural affine scheme
over $R$ associated with $V$, see~\cite[Exp. I, \S 4.6]{SGA} or~\cite[Ch.~IV, \S~1, 1.6]{MilneAGS}.
Any morphism of $R$-schemes $W(V_1)\to W(V_2)$
is determined by an element $f\in\opn{Sym}^*(V_1^\vee)\otimes_R V_2$, where $\opn{Sym}^*$ denotes the symmetric algebra,
and $V_1^\vee$ denotes the dual module of $V_1$. If $f\in\opn{Sym}^d(V_1^\vee)\otimes_R V_2$,
we say that the corresponding morphism is homogeneous of degree $d$.
By abuse of notation, we also write $f:V_1\to V_2$ and call it \emph{degree $d$
homogeneous polynomial map from $V_1$ to $V_2$}. In this context, one has
$$
f(r v)=r^d f(v)
$$
for any $v\in V_1$ and $r\in R$. Moreover, for any $R$-algebra $A$ the induced map $V_1\otimes_RA\to V_2\otimes_RA$,
$v\otimes a\mapsto f(v)\otimes a^d$, also is denoted by $f$.

In this section $P$ denotes a parabolic subgroup of $G$ and $L=L_P$ stands for its Levi subgroup.

\begin{lem}\cite[Lemma 3.9]{StavMulti}\label{lem:relschemes}.
In the setting of Lemma~$\ref{lem:relroots}$, for any $\alpha\in\Phi_P=\Phi(S,G)$ there exists a closed
$S$-equivariant  embedding of $R$-schemes
$$
X_\alpha\colon W\bigl(\Lie(G)_\alpha\bigr)\to G,
$$
satisfying the following condition.

\begin{itemize}
\item[\bf{($*$)}]
Let $A$ be an $R$-algebra such that $G_A$ is split with
respect to a maximal split $A$-torus $T\subseteq L_{A}$. Let $e_\delta$,
$\delta\in\Phi$, be a Chevalley basis of $\Lie(G_A)$, adapted to $T$ and $P$, and
$x_\delta\colon\mathbb G_a\to G_{A}$, $\delta\in\Phi$, be the associated
system of 1-parameter root subgroups
{\rm(}e.g. $x_\delta=\exp_\delta$ of~\cite[Exp. XXII, Th. 1.1]{SGA}{\rm)}.
Let
$$
\pi:\Phi=\Phi(T,G_A)\to\Phi_P\cup\{0\}
$$
be the natural projection.
Then for any
$u=\hspace{-8pt}\sum\limits_{\delta\in\pi^{-1}(\alpha)}\hspace{-8pt}a_\delta e_\delta\in\Lie(G_{R})_\alpha$
%\hspace{-1ex}\sum_{\substack{-2n \leq i \leq 2n \\[.7ex] i\neq 0}}
one has
\begin{equation*}
X_\alpha(u)=
\Bigl(\prod_{\delta\in\pi^{-1}(\alpha)}\hspace{-8pt}x_{\delta}(a_\delta)\Bigr)\cdot
\prod_{i\ge 2}\Bigl(\prod_{\theta\in \pi^{-1}(i\alpha)}
\hspace{-8pt}x_\theta(p^i_{\theta}(u))\Bigr),
\end{equation*}
where every $p^i_{\theta}:\Lie(G_A)_\alpha\to A$ is a homogeneous polynomial map of degree $i$,
and the products over $\delta$ and $\theta$ are taken in any fixed order.
\end{itemize}
\end{lem}

\begin{defn}
The images of closed embeddings $X_\alpha$, $\alpha\in\Phi_P$, satisfying the statement of Lemma~$\ref{lem:relschemes}$,
are called \emph{relative root subschemes of $G$ with respect to the parabolic subgroup $P$} and will be denoted by
$\X_\alpha$.
\end{defn}

Relative root subschemes of $G$ with respect to $P$, actually,
depend on the choice of a Levi subgroup $L_P$, but their essential properties stay the same,
so we usually omit $L_P$ from the notation.

Set $V_\alpha=\Lie(G)_\alpha$ for short.
We will use the following properties of relative root subschemes.

\begin{lem}\label{lem:rootels}\cite[Theorem 2, Lemma 6, Lemma 9]{PetStavIso}
Let $X_\alpha$, $\alpha\in\Phi_P$, be as in Lemma~$\ref{lem:relschemes}$.
Then

(i) There exist degree $i$ homogeneous polynomial maps $q^i_\alpha:V_\alpha\oplus V_\alpha\to V_{i\alpha}$,
$i>1$, such that for any $R$-algebra $A$ and for any
$v,w\in V_\alpha\otimes_RA$ one has
\begin{equation}\label{eq:sum}
X_\alpha(v)X_\alpha(w)=X_\alpha(v+w)\prod_{i>1}X_{i\alpha}\left(q^i_\alpha(v,w)\right).
\end{equation}

(ii) For any $g\in L(R)$, there exist degree $i$ homogeneous polynomial maps
$\varphi^i_{g,\alpha}\colon V_\alpha\to V_{i\alpha}$, $i\ge 1$, such that for any $R$-algebra $A$ and for any
$v\in V_\alpha\otimes_R A$ one has
$$
gX_\alpha(v)g^{-1}=\prod_{i\ge 1}X_{i\alpha}\left(\varphi^i_{g,\alpha}(v)\right).
$$

(iii) \emph{(generalized Chevalley commutator formula)} For any $\alpha,\beta\in\Phi_P$
such that $m\alpha\neq -k\beta$ for all $m,k\ge 1$,
there exist polynomial maps
$$
N_{\alpha,\beta,i,j}\colon V_\alpha\times V_\beta\to V_{i\alpha+j\beta},\ i,j>0,
$$
homogeneous of degree $i$ in the first variable and of degree $j$ in the second
variable, such that for any $R$-algebra $A$ and
for any $u\in V_\alpha\otimes_RA$, $v\in V_\beta\otimes_RA$ one has
\begin{equation}\label{eq:Chev}
[X_\alpha(u),X_\beta(v)]=\prod_{i,j>0}X_{i\alpha+j\beta}\bigl(N_{\alpha,\beta,i,j}(u,v)\bigr)
\end{equation}

(iv) For any subset $\Psi\subseteq X^*(S)\sm\{0\}$ that is closed under addition,
the morphism
$$
X_\Psi\colon W\Bigl(\,\bigoplus_{\alpha\in\Psi}V_\alpha\Bigr)\to U_\Psi,\qquad
(v_\alpha)_{\alpha\in\Psi}\mapsto\prod_\alpha X_\alpha(v_\alpha),
$$
where the product is taken in any fixed order,
is an isomorphism of schemes.
\end{lem}

Apart from the above properties of relative root subschemes we will use the following Lemma,
which appeared first as~\cite[Lemma 10]{PetStavIso} and in a slighly stronger form in~\cite[Lemma 2]{LuzStav}.
Note that in both cases the original statements erroneously claim that the image $\im(N_{\alpha,\beta,1,1})$
(respectively, the sum of images in (2)\,) equals $V_{\alpha+\beta}$, while in reality
the respective proofs establish only that it generates $V_{\alpha+\gamma}$
as an $R$-module, and hence as an abelian group. The correct, weaker statement is as follows.
One can easily check that it is still enough
for all the applications in~\cite{PetStavIso,LuzStav,StavSerreConj}.

\begin{lem}\label{lem:const}
Consider $\alpha,\beta\in\Phi_P$ satisfying $\alpha+\beta\in\Phi_P$ and $m\alpha\neq -k\beta$ for any $m,k\ge 1$.
Denote by $\Phi^0$ an irreducible component of $\Phi$ such that $\alpha,\beta\in\pi(\Phi^0)$.

{\rm (1)} In each of the following cases,
$\im(N_{\alpha,\beta,1,1})$ generates $V_{\alpha+\beta}$ as an abelian group:

\quad {\rm (a)} structure constants of $\Phi^0$ are invertible in $R$ (for example,
if $\Phi^0$ is simply laced);

\quad {\rm (b)} $\alpha\neq \beta$ and $\alpha-\beta\notin\Phi_P$;

\quad {\rm (c)} $\Phi^0$ is of type $B_l$, $C_l$, or $F_4$,
and $\pi^{-1}(\alpha+\beta)$ consists of short roots;

\quad {\rm (d)} $\Phi^0$ is of type $B_l$, $C_l$, or $F_4$, and there exist long roots
$\alpha\in\pi^{-1}(\alpha)$, $\beta\in\pi^{-1}(\beta)$ such that $\alpha+\beta$ is a root.

{\rm (2)} If $\alpha-\beta\in\Phi_P$ and $\Phi^0\neq G_2$, then\;
$\im(N_{\alpha,\beta,1,1})$,\, $\im( N_{\alpha-\beta,2\beta,1,1})$, and\,
$\im\bigl(N_{\alpha-\beta,\beta,1,2}(-,v)\bigr)$ for all $v\in V_\beta$ together
generate $V_{\alpha+\beta}$ as an abelian group. Here we assume
$\im(N_{\alpha-\beta,2\beta,1,1})=0$ if $2\beta\notin\Phi_P$.
\end{lem}

%%==========================================================
\section{Reduction to extraction of unipotents}

In this section we establish two important reduction statements.
Let $P$ be a proper parabolic $R$-subgroup of $G$,
$\Phi_P=\Phi_{J,\Gamma}$ the system of relative roots for $P$, and  $\X_{\alpha}$, $\alpha\in\Phi_P$,
the relative root subschemes of $G$ with respect to $P$, which exist by Lemmas~\ref{lem:relroots} and~\ref{lem:relschemes}.
Recall that for any ideal $\q$ of $R$ we set $U_{P}(\q)=U_{P}(R)\cap G(R,\q)$ and
denote by $E_P(R,\q)$ the normal closure of $\langle U_{P^+}(\q),U_{P^-}(\q)\rangle$ in $E_P(R)$.

First we prove the standard commutator formulas, which are used for the reduction and complements the
normal structure theorem. The condition for inclusions in the following statement is the same as for
normality of the elementary subgroup in Theorem~\ref{thm:EE}. It is easy to see that this condition
for a ring $R$ is inherited by any $R$-algebra.
%% shall we explain this???

\begin{lem}\label{lem:E_P}
Let $\q$ be an ideal of a ring $R$.
Assume that for every maximal ideal $\m$ of $R$ every normal semisimple
subgroup of $G_{R_\m}$ contains $(\mathbb G_{m,R_\m})^2$. Then $E_P(R,\q)$ does not depend on
a strictly proper parabolic subgroup $P$ and is normal in $G(R)$ and
$$
[G(R,\q),E(R)]\le E_P(R,\q).
$$
\end{lem}

\begin{proof}
Let $P$ and $Q$ be two strictly proper parabolic subgroups of $G$.
To prove the first assertion in suffices to show that $X_\alpha(rv)\in E_Q(R,\q)$ for any
$\alpha\in\Phi_P$, $r\in\q$, and $v\in V_\alpha$. Consider an element
$X_\alpha(tv)\in E_P(R[t],tR[t])$, where $t$ is an independent variable.
Since $tR[t]$ is a splitting ideal, by the splitting priciple for isotropic groups~\cite[Lemma 4.1]{StavSerreConj}
we have $E_P(R[t],tR[t])=G(R[t],tR[t])\cap E_P(R[t])$. The latter does not depend on the choice of a
parabolic subgroup $P$ by the main theorem of~\cite{PetStavIso}. Thus
$X_\alpha(tv)\in E_Q(R[t],tR[t])$. Specializing
$t$ to $r$ we get the required inclusion. The normality of $E(R,\q)$ follows from this inclusion
applied to conjugate parabolic subgroups, and the normality of $E(R)$.

The proof of the last statement is borrowed from~\cite[Proposition 5.1]{StepUniloc}.
Let $A=R[G]$ be the affine algebra of the scheme $G$, $I$ its fundumental ideal, $g$ the generic
element of $G$, and $r\in R$. We identify elements of $R$ and $G(R)$ with their the canonical images in
$A$ and $G(A)$ respectively. Consider the commutator
$h=[X_\alpha(r),g]\in G(A)$. Since the elementary subgroup is normal, we have
$h\in E(A)$. On the other hand, $h$ vanishes modulo $I$ and
$I$ is a splitting ideal, i.\,e. $A=R\oplus I$ as additive groups.
The splitting principle implies that
$h\in E(A)\cap G(A,I)=E(A,I)$.

Now, let $b\in G(R,\q)$.
We have
$$
b(h)=[X_\alpha(r),b]\in b\bigl(E(A,I)\bigr)=E\bigl(b(A),b(I)\bigr)\le E\bigl(R,\q\bigr)
$$
by Lemma~\ref{lem:GRq}. Thus, $[G(R,\q),E(R)]\le E(R,\q)$.
\end{proof}

In the setting of Lemma~\ref{lem:E_P} we denote $E_P(R,\q)$ by $E(R,\q)$.

\begin{defn}
We say that the normal structure of the group $G(R)$ is \emph{standard} if for each subgroup $H\le G(R)$
normalized by $E(R)$, there exists a unique ideal $\q$ of $R$ such that
$$E(R,\q)\le H\le C(R,\q).$$
\end{defn}

In order to secure the existence of relative roots and relative root subschemes as in Lemmas~\ref{lem:relroots} and~\ref{lem:relschemes},
assume from now on that
$G$ is a reductive group scheme over a {\bf connected} commutative ring $R$.
We are going to deduce the existence of the standard normal structure from the fact that any $H$
as above contains an elementary root unipotent.
Our starting point is the following theorem established by the first author.

\begin{thm}[{\cite[Theorem~2]{StavCongruence}}]\label{thm:cong-N}
Assume that the structure constants of the absolute root system $\Phi$ of $G$ are invertible in $R$,
and for any maximal ideal $\m$ of $R$ every semisimple normal subgroup of $G_{R_\m}$ contains
$({{\mathbb G}_m}_{,R_\m})^2$.
Let $P$ be a strictly proper parabolic $R$-subgroup of $G$.
Then for any normal subgroup $N\le E(R)$ there exists an ideal
$\q =\q(N)$ in $R$ such that $N\cap\X_\alpha(R)=X_\alpha(\q V_\alpha)$
for any $\alpha\in\Phi_P$.
\end{thm}

The following equalities are called the standard commutator formulas.

\begin{prop}\label{prop:stand-comm-formula}
In the setting of Theorem~$\ref{thm:cong-N}$ we have
$$
[C(R,\q),E(R)]=[G(R,\q),E(R)]=[G(R),E(R,\q)]=[E(R),E(R,\q)]=E(R,\q).
$$
\end{prop}

\begin{proof}
First we prove the inclusion $E(R,\q)\le[E(R),E(R,\q)]$.
By Theorem~\ref{thm:cong-N} it is enough to show that
$X_{\tilde\alpha}(V_{\tilde\alpha}\otimes_R \q)\subseteq [E(R),E(R,\q)]$, where $\tilde\alpha\in\Phi_P^+$ is a maximal root.
Let $\beta\in\Phi_P^+$ be a simple relative root such that $\tilde\alpha-\beta\in\Phi_P$. Then by Lemma~\ref{lem:const}
and the generalized Chevalley commutator formula
$X_{\tilde\alpha}(V_{\tilde\alpha}\otimes_R \q)$
is generated by all commutators $[X_\beta(u),X_{\tilde\alpha-\beta}(v)]$,
where $v\in V_{\tilde\alpha-\beta}$, $u\in V_\beta\otimes_R\q$.

In view of Lemma~\ref{lem:E_P} it remains to prove that $[C(R,\q),E(R)]=E(R,\q)$.
The proof or this formula is an easy group theoretical excersise,
we reproduce it for the sake of completeness.
By the definition of $C(R,\q)$ the left hand side of the formula is contained
in $G(G,\q)$ and we have just proven that it contains $E(R,\q)$.
Therefore, $\bigl[C(R,\q),E(R)],E(R)\bigr]=E(R,\q)$. Now the result follows from
Lemma~1 of~\cite{StepPolynormal}, which asserts that $\bigl[[H,D],D\bigr]=[H,D]$
if $D$ is a perfect subgroup and  $\bigl[[H,D],D\bigr]$ is normal in $H$.
\end{proof}

Using Theorem~\ref{thm:cong-N}, we prove the following statement. The idea of this reduction goes back
to~\cite{Bass64}.

%\begin{aside}
%It seems that if every irreducible component of $\Phi_P$ contains more than 2 distinct roots,
%then each semisimple normal subgroup of $G$ contains $(\mathbb G_m)^2$, which of course implies
%one of the conditions of Theorem~\ref{thm:cong-N}. If it is true then we should at least mention
%it.
%\end{aside}

\begin{prop}\label{lem:reduction}
In the setting of Theorem~$\ref{thm:cong-N}$,
suppose further that every irreducible component of $\Phi_P$ contains more than 2 distinct
roots. If for any ideal $\q$ of $R$ any
noncentral subgroup $H\le G(R/\q)$ normalized by $E(R/\q)$
contains a nontrivial root unipotent element $X_\alpha(u)$, $\alpha\in\Phi_P$, $u\in V_\alpha\otimes_R R/\q\sm\{0\}$,
then the normal structure of the group $G(R)$ is standard.
\end{prop}

To prove Proposition~\ref{lem:reduction} we need the following technical statement.

\begin{lem}\label{lem:ABe}
Let $\alpha,\beta\in\Phi_P$ be two
relative roots such that $\alpha+\beta\in\Phi_P$ and $n\alpha\neq -m\beta$ for all $n,m\ge 1$.
Assume, moreover, that $\alpha-\beta\notin\Phi_P$ or the structure constants of the absolute root system
$\Phi$ of $G$ are invertible in $R$. Take $u\in V_\beta\sm\{0\}$. Any generating system
$e_1,\ldots,e_n$ of the $R$-module $V_\alpha$ contains an element $e_i$ such that $N_{\alpha,\beta,1,1}(e_i,u)\neq 0$.
\end{lem}

\begin{proof}
%%Assume that $N_{\alpha,\beta,1,1}(e_i,u)=0$ for all $1\le i\le n$.
Consider an affine fpqc-covering $\coprod \opn{Spec} R_\tau\to \opn{Spec} R$ that splits $G$. There is a member
$R_\tau=A$ of this covering such that the image of $X_\beta(u)$ under $G(R)\to G(A)$ is non-trivial.
Write
\begin{equation*}
X_\beta(u)=\prod_{\pi(\nu)=\beta}x_{\nu}(a_{\nu})\cdot
\prod_{i\ge 2}\prod_{\pi(\nu)=i\beta}x_{\nu}(c_{\nu}),
\end{equation*}
where $\pi:\Phi\to\Phi_P\cup\{0\}$ is the canonical projection of the absolute root system of $G$ onto the relative one,
$x_\nu$ are root subgroups of the split group $G_A$,
and $a_{\nu}\in R$. Since $X_\beta(u)\neq 0$, the definition of $X_\beta$ implies that there exists
$a_{\nu}\neq 0$.
By~\cite[Lemma 4]{PetStavIso} there exists
a root $\mu\in\pi^{-1}(\alpha)$ such that $\mu+\nu\in\Phi$. Let $v\in V_\alpha\otimes_R A$
be such that
$X_\alpha(v)=x_{\mu}(1)\prod\limits_{i\ge 2}\prod\limits_{\pi(\eta)=i\alpha}x_{\eta}(d_{\eta})$,
for some $d_{\eta}\in R$.
Then the (usual) Chevalley commutator formula implies that $[X_\alpha(v),X_\beta(u)]$ contains in its decomposition a
factor $x_{\mu+\nu}(\eps a_{\nu})$, where $\eps$ is a structure constant
of $\Phi$.

If $\alpha-\beta\notin\Phi_P$, then $\eps=\pm 1$, otherwise $\eps$ is invertible by assumptions.
Hence $N_{\alpha,\beta,1,1}(v,u)\neq 0$. Since $N_{\alpha,\beta,1,1}(v,u)$ is linear in the first argument,
this implies the result.
\end{proof}

\begin{proof}[Proof of Proposition~$\ref{lem:reduction}$]
Let $N$ be a subgroup of $G(R)$ normalized by $E(R)$.
If $N\le \Center\bigl(G(R)\bigr)$, there is nothing to prove. Otherwise by our assumption and
Theorem~\ref{thm:cong-N} there is an ideal $\q\neq 0$ of $R$ such that $N\cap\X_\alpha(R)=X_\alpha(\q V_\alpha)$
for all $\alpha\in\Phi_P$. Then, clearly, $E_P(R,\q)\le N$. If $N$ is not contained in $C(R,\q)$, then
$\rho_{\q}(N)$ is a non-central subgroup of $G(R/\q)$ normalized by $E(R/\q)$, and hence by the same token we have
$N\ge E_P(R/\q,\q')$ for some ideal $\q'\neq 0$ of $R/\q$. Let $\tilde\alpha\in\Phi_P^+$ be a maximal root,
and let $\beta\in\Phi_P^+$ be a simple relative root such that $\tilde\alpha-\beta\in\Phi_P$. Pick
$u\in V_{\tilde\alpha-\beta}\sm\{0\}$ such that $\rho_\q(u)\in V_{\tilde\alpha-\beta}\otimes_R \q'\sm\{0\}$
and
$X_{\tilde\alpha-\beta}(\rho_\q(u))\in \rho_q(N)$. Then there is $h\in G(R,\q)$ such that
$X_{\tilde\alpha-\beta}(u)h\in N$.
By Lemma~\ref{lem:ABe} for any generating system
$e_1,\ldots,e_n$ of $V_\beta$
there is $1\le i\le n$ such that $N_{\tilde\alpha-\beta,\beta, 1,1}(\rho_\q(u),\rho_\q(e_i))\neq 0$,
and hence $N_{\tilde\alpha-\beta,\beta, 1,1}(u,e_i)\not\in \q V_{\tilde\alpha}$.
By Lemma~\ref{lem:E_P} one has
$$
[X_{\tilde\alpha-\beta}(u)h,X_\beta(e)]=[h,X_\beta(e)]^{X_{\tilde\alpha-\beta}(-u)}\cdot
X_{\tilde\alpha}(N_{\tilde\alpha-\beta,\beta, 1,1}(u,e_i))\in
E_P(R,\q)\cdot X_{\tilde\alpha}(N_{\tilde\alpha-\beta,\beta, 1,1}(u,e_i)).
$$
Hence $X_{\tilde\alpha}(N_{\tilde\alpha-\beta,\beta, 1,1}(u,e_i))\in N$. However, this contradicts the choice of $\q$.
\end{proof}

Thus, we have reduced the proof of the normal structure theorem to extraction of unipotents from $E(R)$-normalized subgroups.
In the following section we show that one can always extract a nontrivial root unipotent from an element that belongs to
a proper parabolic subgroup.

%%==========================================================
\section{Inside a parabolic subgroup}

As before, assume that $G$ is a reductive group scheme over a connected ring $R$.
We start with definitions of certain normal subgroups of the unipotent radical of a parabolic subgroup $P$ of $G$.
Choose a set $\Pi_P$ of simple relative roots and fix a total order on the relative root lattice
$\Z\Phi_P$ in the following way. First, order $\Pi_P$ arbitraraly. Then set $\alpha\ge\beta$
if the coordinate columns $\alpha_{\Pi_P}\ge\beta_{\Pi_P}$ with respect to
the lexicographical order.
Note that for this order we have
\begin{equation}\label{eq:order}
\alpha>\gamma\iff-\alpha<-\gamma\iff\alpha+\beta>\gamma+\beta
\end{equation}
for all $\alpha,\beta,\gamma\in\Z\Phi_P$.

For $\alpha\in\Z\Phi_P$ put
$$
\Phi_P^{>\alpha}=\{\gamma\in\Phi_P\mid\gamma>\alpha\}\text{ and }
\Phi_P^{<\alpha}=\{\gamma\in\Phi_P\mid\gamma<\alpha\}.
$$
If $\alpha>0$, then the set $\Phi_P^{>\alpha}$ is a closed set of roots and we denote by $U_P^{>\alpha}$
the subgroup scheme corresponding to $\Phi_P^{>\alpha}$ (see Lemma~\ref{lem:T-P}).
Similarly, if $\alpha<0$, then $U_P^{<\alpha}$ is the subgroup scheme corresponding to the closed set of roots
$\Phi_P^{<\alpha}$.
If $\alpha$ is the largest positive root of a given height $m$, then we denote $U_P^{>\alpha}$ by $U_P^{(m+1)}$.
In other words, $U_P^{(m+1)}(R)$ is the subgroup of $U_P(R)$ generated by all root subgroups $\X_\gamma(R)$ with
$\opn{ht}\gamma>m$.

\begin{lem}\label{lem:>alpha}
Let $\alpha,\beta,\alpha+\beta\in\Phi_P^+$, then
the subgroup $U_P^{>\alpha}(R)$ is normal in $P(R)$ and, moreover,
$$[U_P^{>\alpha}(R),\X_\beta(R)]\le U_P^{>\alpha+\beta}(R).$$

Similarly, for $\alpha,\beta,\alpha+\beta\in\Phi_P^-$ the subgroup $U_P^{<\alpha}(R)$ is normal in $P^-(R)$ and
$$[U_P^{<\alpha}(R),\X_\beta(R)]\le U_P^{<\alpha+\beta}(R).$$
\end{lem}

\begin{proof}
We prove the first formula by going down induction on $\alpha$.
If $\alpha$ is the maximal root, then $U_P^{>\alpha}=\{1\}$. Otherwise,
let $\gamma$ be the smallest root greater than $\alpha$.
By Lemma~\ref{lem:rootels}(iv) $U_P^{>\alpha}(R)=\X_\gamma(R)U_P^{>\gamma}(R)$.
Then
$$
[U_P^{>\alpha}(R),\X_\beta(R)]=
\conj{\X_\gamma(R)}{[U_P^{>\gamma}(R),\X_\beta(R)]}\cdot[\X_\gamma(R),\X_\beta(R)].
$$
If $\beta>0$, then by induction hypothesis  the first factor is contained in $U_P^{>\gamma+\beta}(R)$
whereas the second one lies in $\X_{\gamma+\beta}(R)U_P^{>\gamma+\beta}(R)$ by generalized Chevalley commutator formula~\eqref{eq:Chev}. Hence $U_P^{>\alpha}(R)$ is normal in $U_P(R)$. With this in mind
for an arbitrary $\beta$ the above arguments show that both factors lie in $U_P^{>\alpha+\beta}(R)$ as required.

Since $P$ is generated by $L_P$ and $U_P$ and $L_P$ normalizes each relative root subgroup, we conclude that
$U_P^{>\alpha}(R)$ is normal in $P(R)$.
The proof of the second formula is essentially the same.
\end{proof}

For the rest of the article we assume that each relative root system is endowed with
a total ordering, satisfying conditions~\eqref{eq:order}.

\begin{lem}\label{lem:InU}
Suppose that the absolute root system of $G$ is irreducible and its structure constants
are invertible in $R$. Let $P$ be a parabolic $R$-subgroup of $G$ with a Levi subgroup $L_P$,
such that $\opn{rank}(\Phi_P)\ge 2$.
For any $h\in U_P(R)\sm\{1\}$, the subgroup
$[h,E(R)]$ contains a root unipotent element $X_{\alpha}(v)$, $\alpha\in\Phi_P^+$,
$v\in V_\alpha\sm\{0\}$.
\end{lem}

\begin{proof}
By Lemma~\ref{lem:rootels}(iv)  we can write $h$ as a product
$$
h=\prod_{\alpha\in\Phi_P^+}X_\alpha(v_\alpha),\quad v_\alpha\in V_\alpha,
$$
where the order of factors coincides with the ascending order of roots. Let $\gamma=\gamma(h)\in\Phi_P^+$ be
the minimal root such that $v_{\gamma}\neq 0$.
If $\gamma$ is the maximal root of $\Phi_P^+$, then $h=X_\gamma(v_\gamma)$.
Since $\opn{rank}(\Phi_P)\ge 2$, there is a simple root $\beta\in\Phi_P^+$ such that $\gamma-\beta$ is a root
and all roots of shape $i\gamma-j\beta$ are positive.
By Lemma~\ref{lem:ABe} there is $v\in V_\beta$ such that
$N_{\gamma,-\beta,1,1}(v_{\gamma},v)\neq 0$. Then $h'=[h,X_{-\beta}(v)]\in U_P(R)\sm\{1\}$
and $\gamma(h')$ is not maximal. Since $[h',E(R)]\le[h,E(R)]$ (see remark after formula~\eqref{xyzz-1}),
we may assume from the beginning that $\gamma(h)$ is not maximal.

Under this assumption write $h=X_\gamma(v_\gamma)g$ for some $g\in U_P^{>\gamma}(R)$.
There is a root $\beta\in\Phi_P^+$ such that $\gamma+\beta\in\Phi_P^+$.
By Lemma~\ref{lem:ABe} there is $v\in V_\beta$ such that $N_{\gamma,\beta,1,1}(v_{\gamma},v)\neq 0$.
The commutator
$$
c=[h,X_\beta(v)]=\conj{X_\gamma(v_\gamma)}{[g,X_\beta(v)]}\cdot [X_\gamma(v_\gamma),X_\beta(v)].
$$
By Lemma~\ref{lem:>alpha} the first factor belongs to $U_P^{>\gamma+\beta}(R)$ and by the generalized Chevalley commutator formula
the second one is equivalent to $X_{\gamma+\beta}\bigl(N_{\gamma,\beta,1,1}(v_{\gamma},v)\bigr)$
modulo $U_P^{>\gamma+\beta}(R)$. Thus, $c\in U_P(R)\cap[h,E(R)]$ and $\gamma(c)=\gamma(h)+\beta>\gamma(h)$.
Therefore, we may proceed by
induction to arrive at the case, where $\gamma(c)$ is the maximal root, i.\,e. $c=X_{\gamma(c)}(u)\in[h,E(R)]$
for some $u\in V_{\gamma(c)}\sm\{0\}$.
\end{proof}

\begin{lem}\label{lem:Levi-repr}
Let $P$ be a stricly proper parabolic subgroup of $G$ with a Levi subgroup
$L_P$. Let $z\in L_P(R)$
be such that $[z,U_P(R)]\subseteq U_P^{(2)}(R)$. Then $z\in\Center(G)(R)$.
\end{lem}

\begin{proof}
Replace $G$ by $G^{ad}=G/\Center(G)$. Then we need to show that $z=1$.
By Lemma~\ref{lem:rootels} for any $\alpha\in\Phi_P$ there exists a set of degree $i$ homogeneous polynomial maps
$\varphi^i_{z,\alpha}: V_\alpha\to V_{i\alpha}$, $i\ge 1$, such that for
any commutative
$R$-algebra $A$ and
any $w\in V_\alpha\otimes_R A$ one has
$$
zX_\alpha(w)z^{-1}=\prod_{i\ge 1}X_{i\alpha}(\varphi^i_{z,\alpha}(w)).
$$
If $\alpha\in\Phi_P^+$ has height $>1$, then $zX_\alpha(w)z^{-1}\in U_P^{(2)}(A)$ automatically. Thus, the inclusion
$[z,U_P(R)]\subseteq U_P^{(2)}(R)$ is equivalent to $\varphi^1_{z,\alpha}(v)=v$ for any $\alpha$ of height 1 and any
$v\in V_\alpha$.
Since $\varphi^1_{z,\alpha}$ is linear, $\varphi^1_{z,\alpha}(v)=v$ for any $v\in V_\alpha$
implies $\varphi^1_{z,\alpha}(w)=w$ for any $w\in V_\alpha\otimes_R A$.

In other words, the conjugation action of $L_P$ on $U_P$
induces a linear group scheme representation $\rho:L_P\to \GL(V)$, where
$V=\bigoplus_{\opn{ht}\alpha=1} V_\alpha$, via the natural isomorphism $U_P/U_P^{(2)}\cong W(V)$.
By our assumption, we have $\rho(z)=\sum_{\opn{ht}\alpha=1}\varphi^1_{z,\alpha}=\mathrm{id}_V$,
i.e. $z\in\ker\rho(R)$. Clearly, $\ker\rho$ is a closed subgroup scheme of  $L_P$.
We intend to show that it is trivial.

Since $L_P$ is a group scheme locally of finite type over $R$, $\ker\rho$ is also locally of finite type. Then
by~\cite[Exp. $\mathrm{VI_B}$, Corollaire 2.10]{SGA} in order to prove that $\ker\rho=1$, it is enough to prove
that $(\ker\rho)_{k(\mathfrak{p})}=1$ for every $\mathfrak{p}\in\opn{Spec} R$. For the latter it suffices to
show that the scheme $(\ker\rho)_F$ is trivial, where $F=\overline{k(\mathfrak{p})}$ is the algebraic closure of
$k(\mathfrak{p})$.

Thus, without loss of generality we may assume that $G$ is a split reductive group over an algebraically closed field $R=F$,
and $P$ is a standard parabolic subgroup of $G$ with respect to a split maximal $F$-torus $T\le L_P$.
We denote by $\Phi$ the root system of $G$ over $F$, and $P$ has type $J$, as in Lemma~\ref{lem:relroots}.
Every normal subgroup of $L_P(F)$ is either central in $L_P(F)$ or contains the elementary subgroup, corresponding to a whole irreducible component
of the root system of $L_P(F)$ (which is non-trivial unless $L_P=T$), e.g. by~\cite{Tits64}.
Clearly, $\ker\rho(F)$ does not contain the elementary subgroup of this form.
Indeed, since the sum of all simple roots of $\Phi$ is a root,
for any simple root $\beta\not\in J$ there exists a root $\alpha\in\pi^{-1}(\opn{ht}\alpha=1)$ such that
$\alpha+\beta\in\Phi$ and
$\alpha-\beta\not\in\Phi$; then $[x_\beta(1),x_\alpha(1)]\neq 1$. If $j$ is a maximal positive integer such
that $\alpha+j\beta\in\Phi$, then
$\alpha+j\beta-(-\beta)\notin\Phi$, and hence $[x_{-\beta}(1),x_{\alpha+j\beta}(1)]\neq 1$.
Then $\ker\rho(F)\le\Center(L_P(F))$. One has $\Center(L_P(F))=\Center(L_P)(F)\le T(F)$ (e.g. by~\cite{AbeHurley}),
hence $\ker\rho(F)\le T(F)$.

Similarly, $\Lie(\ker\rho)\le \Lie(T)$. Indeed,
the Lie algebra $\Lie(L_P)$, and hence its subalgebra $\Lie(\ker\rho)$, decomposes into a direct sum of $\Lie(T)$
and 1-dimensional
root eigenspaces of $T$. Since $\ker\rho$ is a normal subgroup of $L_P$,
its Lie algebra is invariant under the adjoint action of $L_P(F)$ on $\Lie(L_P)$, and in particular
under the action of the Weyl group of $L_P(F)$.
Hence $\Lie(\ker\rho)$ either contains no root subspaces at all, or contains all root subspaces corresponding to one
irreducible component of the root system of $L_P$. One proves similarly to the above that it cannot contain such
a set of root subspaces.

Let $A$ be any $F$-algebra, then
any $x\in T(A)$ acts on root elements $x_\delta(a)$, $\delta\in\Phi$, $a\in A$, by means
of a character $\chi$ of the root lattice of $G$ (recall that $G$ is adjoint) with values in $A^\times$.
Since any $\alpha\in\Phi_P^+$ is a sum of simple
roots $\beta_i\in\Phi_P^+$, by~\cite[Lemma 4]{PetStavIso} any $\delta\in\pi^{-1}(\alpha)$ is a sum of roots
$\delta_i\in\pi^{-1}(\beta_i)$. If $x\in\ker\rho(A)$, we have $\chi(\delta_i)=1$ for all $i$, hence $\chi(\pm\delta)=1$
for any $\delta\in\pi^{-1}(\Phi_P^+)$, and thus $x=1$. This implies that both $\ker\rho(F)=1$ and $\Lie(\ker\rho)=0$.
In particular, $\dim(\ker\rho)=\dim(\Lie(\ker\rho))$, which implies that $\ker\rho$ is smooth. Since
$\ker\rho(F)=1$, we have $\ker\rho=1$, as required.
\end{proof}

For a relative root $\gamma\in\Phi_P$ denote by $U_{(\gamma)}$ the unipotent subgroup generated by
$\X_{i\gamma}$ for all $i\in\N$ such that $i\gamma\in\Phi_P$. In other words, if
we set $(\gamma)=\N\gamma\cap\Phi_P$, then $U_{(\gamma)}$ is the group scheme introduced in
Lemma~\ref{lem:T-P}.

\begin{cor}\label{cor:levi-cent-beta}
Let $P$ be a not necessarily strictly proper parabolic subgroup of $G$. Let $\beta\in\Phi_P^+$ be a simple root.
If $g\in L_P(R)$ satisfies $[g,\X_\beta(R)]\le U_P^{>\beta}(R)$, then $[g,U_{(\beta)}(R)]=1$ and $[g,U_{(-\beta)}(R)]=1$.
\end{cor}
\begin{proof}
Let $H=U_{\Z\beta\cap\Phi_P}$, the reductive subgroup scheme of $G$ that exists by Lemma~\ref{lem:T-P}. By the same Lemma,
it has two opposite
proper parabolic subgroups $Q=L_PU_{(\beta)}$ and $Q^-=L_QU_{(-\beta)}$ with the common Levi subgroup $L_Q=L_P$.
The assumption of the Lemma then implies $[g,U_Q(R)]\le U_Q^{(2)}(R)$.

Note that it is enough to prove the claim for $H^{ad}$ instead of $G$. Indeed, if it holds for $H^{ad}$, then
$[g,U_{(\pm\beta)}]\le C(H)(R)$. However, $[g,U_{(\pm\beta)}(R)]\le U_{(\pm\beta)}(R)$ and
$C(H)(R)\cap U_{(\pm\beta)}(R)=1$, hence $[g,U_{(\pm\beta)}(R)]=1$.

Assume that $G=H^{ad}$. Since $R$ is connected, by~\cite[Exp. XXIV, Proposition 5.10]{SGA} $G$ is a finite direct product of
adjoint semisimple
reductive $R$-groups $G_i$, $1\le i\le n$, such that every $G_i$ itself does not have any non-trivial normal
semi-simple subgroups. Consider the action of the torus $S\le C(L_P)$ constructed in Lemma~\ref{lem:relroots} on $G_i$.
If this action is non-trivial, then the projection of $g$ into $G_i(R)$ is trivial by Lemma~\ref{lem:Levi-repr}, since
the projection of $Q$ is a proper parabolic subgroup of $G_i$. Hence $g$ belongs to the product of groups $G_i(R)$
such that $S$ centralizes $G_i$, and, in particular, $G_i\le L_P$. Then, since $G_i$ is normal in $G$,
$$[G_i(R),U_{(\beta)}(R)]\le G_i(R)\cap U_{(\pm \beta)}(R)\le L_P(R)\cap  U_{(\pm \beta)}(R) =1,$$
and hence $g$ centralizes $U_{(\pm \beta)}(R)$.

\end{proof}

\begin{lem}\label{lem:InP}
Suppose that the absolute root system of $G$ is irreducible and its structure constants
are invertible in $R$. Let $P$ be a parabolic $R$-subgroup of $G$ with a Levi subgroup $L_P$,
such that $\opn{rank}(\Phi_P)\ge 2$.
Let $H$ be a subgroup of $G(R)$ normalized by the elementary subgroup $E(R)$.
Suppose that $H\cap P(R)$ is not contained in $\Center\bigl(G(R)\bigr)$.
Then $H$ contains a root unipotent element $X_{\alpha}(v)$, $\alpha\in\Phi_P$,
$v\in V_\alpha\sm\{0\}$.
\end{lem}

\begin{proof}
This is proved similarly to the corresponding statement~\cite[Theorem~1]{VavStav} for Chevalley groups.
Let $g\in H\cap P(R)$ be a non-central element.
Write $g=zu$, where $z\in L_P(R)$, $u\in U_P(R)=U_{\Phi_P^+}(R)$.
For any $\alpha\in\Phi_P^+$ and $v\in V_\alpha$, one has
$$
[g,X_{\alpha}(v)]=\conj{z}{[u,X_{\alpha}(v)]}\cdot[z,X_{\alpha}(v)].
$$
Clearly, $\conj{z}{[u,X_{\alpha}(v)]}\in [U_P,U_P](R)$.
If $[z,X_{\alpha}(v)]\not\in [U_P,U_P](R)$, then $1\neq [g,X_{\alpha}(v)]\in U_P(R)\cap H$. Then $H$
contains a root unipotent element by Lemma~\ref{lem:InU}.

Assume that $[z,X_{\alpha}(v)]\in [U_P,U_P](R)$
for any $\alpha\in\Phi_P^+$ and $v\in V_\alpha$. Then $[z,U_P(R)]\le U_P^{(2)}(R)$,
and hence $z\in\Center(G)(R)$ by Lemma~\ref{lem:Levi-repr}.
Then $[g,E_P(R)]=[u,E_P(R)]$. Then
$[u,E_P(R)]\le H$ contains a root unipotent element by Lemma~\ref{lem:InU}.
\end{proof}

%%==========================================================
\section{Extraction from the main Gauss cell}

As before, $G$ is a reductive group scheme over $R$
and $P$ is a proper parabolic $R$-subgroup of $G$.
Consider the $R$-subscheme of $G$
$$
\Omega_P=U_PL_PU_{P^-}.
$$
This is an open subscheme isomorphic as an $R$-scheme
to the direct product $U_P\times L_P\times U_{P^-}$, see~\cite[Exp. XXVI, Remarque 4.3.6]{SGA}.
We call $\Omega_P$ \emph{the main Gauss cell} associated with $P$.

In this section we show how to extract a unipotent from $\Omega_P(R)$.
The following easy lemma will be used several times.

\begin{lem}[{see~\cite[Part 1, 1.5]{Jantzen}}]\label{lem:open-fields}
Let $G$ be an affine scheme over a commutative ring $R$, \ $\Omega$ an open
subscheme of $G$, and $g\in G(R)$. Then $g\in\Omega(R)$ if and only if for all maximal ideals
$\m$ of $R$ one has $\rho_\m(g)\in \Omega(R/\m)$.
\end{lem}

\begin{proof}
By definition of an open set $\Omega=D(\q)$ for some ideal $\q$ of the affine algebra
$A=R[G]$ of the scheme $G$. Then
$$
g\in\Omega\iff g(\q)R=R\iff \rho_\m\bigl(g(\q)\bigr)\ne\{0\}\text{ for all maximal ideals }\m\text{ of }R.
$$
But the latter is equivalent to the condition $\rho_\m(g)\in \Omega(R/\m)$
for all maximal ideals $\m$ of $R$.
\end{proof}

From now and until the end of this section, assume that $R$ is connected.

\begin{lem}\label{lem:u-cent-field}
Let $R$ be a field and $P$ a parabolic subgroup of $G$ over $R$.
If $g\in G(R)$ commutes with $U_P(R)$, then $g\in P(R)$.
\end{lem}

\begin{proof}
Let $Q\le P$ be a minimal parabolic subgroup of $G$ over $R$, let $L_Q$ be its Levi subgroup contained in $L_P$,
and let $Q^-$ be the opposite minimal parabolic subgroup contained in $P^-$. Let $S\le L_Q$ be the maximal split
subtorus of $\Center(L_Q)$. Reduced Bruhat decomposition implies
that $g=uwv$, where $u\in U_Q(R)$, $w\in N_G(S)(R)$,
$v\in (U_Q)_w(R)=\{x\in U_Q(R)\mid wxw^{-1}\in U_{Q^-}(R)\}$, and $u,v$, and the class of $w$ in the Weyl group of $\Phi_Q$
are unique.
We have $w\in L_P(R)$ if and only if
$w$ is a product of elementary reflections $w_{\alpha_i}$ for some simple roots $\alpha_i\in\Phi_Q$
such that $X_{\alpha_i}\le L_P$~\cite[Th\'eor\`eme 5.15, Proposition 5.17]{BorelTits}.

Assume that $w\notin L_P(R)$. Then there is a simple root $\alpha\in\Phi_Q$ outside the root system of $L_P$
such that $w(\alpha)<0$.
%Consider $A=\pi(\alpha)$. Let $e_A\in V_A$ be a vector from the generating set existing by the
%hypothesis of the Lemma such that $x_{\alpha}(\xi)$, $\xi\neq 0$, occurs in the canonic decomposition of $x=X_A(e_A)$
%into a product of elementary root unipotents from $U^+$.
Then $x=X_{\alpha}(\xi)$, belongs to $U_P(R)$ for all $\xi\in V_\alpha$. Since
$[g,x]=1$, we have $x(uwv)=(uwv)x$.
The rightmost factor in the reduced Bruhat decomposition of $x(uwv)=(xu)wv$ equals $v$.
%%-- ABC
On the other hand, $w(vx)w^{-1}\in U_{Q^-}(R)$, i.\,e. $vx\in(U_Q)_w(R)$, which means that the rightmost factor of the
reduced Bruhat decomposition of $(uwv)x=uw(vx)$ equals $vx$. But the latter is in contrary to the uniqness of the decomposition.
Thus, $w\in L_P$ and, since $u,v\in U_Q\le Q\le P$, we get the result.
%%However, since $\alpha$ is a positive root of minimal height, it is clear that the rightmost factor in the Bruhat decomposition
%%of $(uwv)x$ contains  $X_\alpha(\eta+\xi)$ in its canonic decomposition, if $v$ contains $X_\alpha(\eta)$. Therefore,
%%this rightmost factor is distinct from $v$, a contradiction.
%
%%Therefore, $w\in L_P(R)$. Then for any $x\in U_P(R)$ we have $wxw^{-1}\in U_P(R)$,
%%hence by the definition of the Bruhat decomposition $v\in L_P(R)\cap U_Q(R)$. This means that
%%$g=uwv\in U_Q(R)L_P(R)=P(R)$.
%%-- end ABC
\end{proof}

\begin{lem}\label{lem:centr-beta}
Suppose that the structure constants of the absolute root
system of $G$ are invertible in $R$. Let $Q$ be a parabolic $R$-subgroup of $G$,
such that $\Phi_Q$ is irreducible and $\opn{rank}(\Phi_Q)\ge 2$.
Let $\beta\in\Phi_Q$ be a simple relative root.
If $x\in U_{Q^\pm}(R)$ commutes
with $\X_\beta(R)$, then $x=\prod X_\alpha(u_\alpha)$, where $\alpha$ ranges over the set of
roots of $\Phi_Q^\pm$ such that $\alpha+\beta\notin\Phi_Q\cup\{0\}$.
\end{lem}

\begin{proof}
First, consider the case where $x\in U_{Q}(R)$.
Write $x=\prod_{\alpha\in\Phi_Q^+} X_{\alpha}(u_\alpha)$,
where $u_\alpha\in V_{\alpha}$ and the order of factors coincides with the
descending order of roots. Let $\gamma$ be the smallest root such that $\gamma+\beta\in\Phi_Q$ and $u_\gamma\ne0$.
Then $x=x'X_\gamma(u_\gamma)x''$, where $x'\in U_Q^{>\gamma}(R)$ and $x''$ commutes with $\X_\beta(R)$.
By Lemma~\ref{lem:ABe} there is $v\in V_\beta$ such that
$N_{\beta,\gamma,1,1}(v,u_\gamma)=w\neq 0$.
Now, by formula~\eqref{xyzz-1}, Lemma~\ref{lem:>alpha}, and generalized Chevalley commutator formula~\eqref{eq:Chev}
we have
$$
[X_\beta(v),x]=[X_\beta(v),x']\cdot\conj{x'}{[X_\beta(v),X_\gamma(u_\gamma)]}
\equiv X_{\beta+\gamma}(w)\mod U_Q^{>\beta+\gamma}(R).
$$
Hence the commutator is not equal to 1. The contradiction shows
that $\alpha+\beta\notin\Phi_Q$ for all $\alpha\in\Phi^+_Q$ such that $u_\alpha\ne0$.

Second, assume that $x=\prod X_{\alpha}(u_{\alpha})$, where $\alpha$ ranges over the set
$\Phi_Q^-\sm\N\cdot(-\beta)$, $u_{\alpha}\in V_{\alpha}$, and the order of factors
coincides with the ascending order of roots.
In other words, in this case $x$ lies in the unipotent radical of the smallest parabolic subgroup $P$
containing $Q^-$ and $\X_\beta$.
Let $\gamma$ be the largest root such that $\gamma+\beta\in\Phi$ and $u_\gamma\ne0$.
Since $\beta$ is a simple root and $\gamma<0$, the root $\gamma+\beta$ is negative.
Write $x=x'X_\gamma(u_\gamma)x''$, where $x'\in U_Q^{<\gamma}$ and $x''$ commutes with $\X_\beta(R)$.
Similarly to the first pragraph of the proof one shows that the commutator
$[x,X_\beta(v)]\not\equiv 1\mod U_Q^{<\gamma+\beta}$, which implies the result.

Finally, let $x\in U_{Q^-}(R)$ be an arbitrary element. Write $x=x_1x_2$, where $x_1\in U_{Q^-}(R)\cap L_P$
(i.\,e. $x_1$ is a product of root elements corresponding to the roots in $\N(-\beta)$)
and $x_2\in U_P$.
For any $v\in V_\beta$ one has
$$
1=[X_\beta(v),x_1x_2]=[X_\beta(v),x_1]\cdot\conj{x_1}{[X_\beta(v),x_2]},
$$
where $[X_\beta(v),x_1]\in L_P(R)$ and $\conj{x_1}{[X_\beta(v),x_2]}\in U_P(R)$, hence
$$
1=[X_\beta(v),x_1]=[X_\beta(v),x_2].
$$
By the previous case we conclude that $x_2$ is a product of root elements corresponding to $\alpha\in\Phi_Q^-$
such that $\alpha+\beta\notin\Phi_Q$. By definition of $x_2$ we also have $\alpha+\beta\neq 0$.
It remains to show that $x_1=1$.

Assume that $x_1\neq 1$. We can write $x_1=\prod_{i=1}^k X_{-i\beta}(u_{-i\beta})$, $u_{-i\beta}\in V_{-i\beta}$.
Since $\Phi_Q$ is irreducible and its rank $\ge 2$, there exists
a simple relative root $\alpha\in\Phi_Q^+$ such that $\alpha+\beta\in\Phi_Q$. Then $\alpha+i\beta\in\Phi_Q$
for any $i\ge 1$ such that $i\beta\in\Phi_Q$ by Lemma~\ref{lem:adj-simple-roots}.
Let $j\ge 1$ be the smallest natural number such that $u_{-j\beta}\neq 0$. Since $-j\beta-\alpha\in\Phi_Q$,
by Lemma~\ref{lem:ABe} there is
$u\in V_{-\alpha}$ such that $N_{-j\beta,-\alpha,1,1}(u_{-j\beta},u)=w\neq 0$.
Set $y=[x_1,X_{-\alpha}(w)]$. Similarly to the first paragraph of the proof
$y$ equals $X_{-j\beta-\alpha}(w)$ modulo $U_Q^{<-j\beta-\alpha}$.

On the other hand, $\beta-\alpha\notin\Phi_Q$ as $\alpha$ and $\beta$ are simple roots.
Hence $[X_\beta(v),X_{-\alpha}(u)]=1$ for any $v\in V_\beta$. Therefore,
$$
[X_\beta(v),y]=[X_\beta(v),[x_1,X_{-\alpha}(w)]]=1
$$
for any $v\in V_\beta$. By the previous case this implies that $y$ cannot contain
$X_{-j\beta-\alpha}(u)$ with non-zero $u$ in its reduced decomposition, a contradiction. This shows that $x_1=1$.
\end{proof}

\begin{lem}\label{lem:small-levi-b}
Under the hypothesis of Lemma~$\ref{lem:centr-beta}$, let $m\ge 1$ be the maximal positive integer such that
$m\beta\in\Phi_Q$. If $x\in U_{(\beta)}(R)L_Q(R)U_{(-\beta)}(R)$ commutes
with $\X_\beta(R)$, then $x\in\X_{m\beta}(R)L_Q(R)$.
\end{lem}

\begin{proof}
Assume that $x=ahb$ for some $a\in U_{(\beta)}(R)$, $h\in L_Q(R)$, and $b\in U_{(-\beta)}(R)\sm\{1\}$.
Put $b'=hb^{-1}h^{-1}$ and write it in the form $b'=X_{-j\beta}(u)b''$, where $u\in V_{-j\beta}\sm\{0\}$ and
$b''\in U_Q^{<-j\beta}$.
Let $\alpha\neq\beta\in\Phi_Q$ be a simple relative root adjacent to $\beta$.
Since $i\beta-\alpha\notin\Phi_Q$ (coefficients at simple roots have different signs),
the generalized Chevalley commutator formula implies that
$[\X_{-\alpha}(R),U_{(\beta)}(R)]=1$. Hence,  $\X_{-\alpha}(R)$ commutes with $a$.
Since $\X_\beta(R)$ commutes also with $x$, we have $\bigl[[\X_{-\alpha}(R),x^{-1}],\X_\beta(R)\bigr]=1$.
On the other hand, by Lemma~\ref{lem:adj-simple-roots}
$-j\beta-\alpha\in\Phi_Q$. Therefore, by Lemma~\ref{lem:ABe} there is $v\in V_{-\alpha}$ such that
$N_{-\alpha,-j\beta,1,1}(u,v)\neq 0$. Then,
\begin{multline*}
c=[X_{-\alpha}(v),x^{-1}]=[X_{-\alpha}(v),h^{-1}X_{-j\beta}(u)b'a^{-1}]=\\
[X_{-\alpha}(v),h^{-1}]\cdot
[X_{-\alpha}(v),X_{-j\beta}(u)]^h\cdot
\conj{h^{-1}X_{-j\beta}(u)}{[X_{-\alpha}(v),b']}
\end{multline*}
By Lemma~\ref{lem:rootels} the first factor belongs to $U_{(-\alpha)}$, whereas the second one is equal to
$X_{-\alpha-j\beta}(w)$ modulo $U_Q^{<-\alpha-j\beta}$ for some $w\in V_{-\alpha-j\beta}\sm\{0\}$.
By Lemma~\ref{lem:>alpha} the last factor lies in $U_Q^{<-\alpha-j\beta}$.
Thus, $c$ contains $X_{-\alpha-j\beta}(w)$ in its reduced decomposition and by Lemma~\ref{lem:centr-beta}
cannot commute with $X_\beta(R)$ as $(-\alpha-j\beta)+\beta\in\Phi_Q$.
The contradiction shows that $b=1$, and $x=ah$, where $a\in U_{(\beta)}(R)$ and $h\in L_Q(R)$.
Then for every $v\in V_\beta$ we have
$$
1=[X_{\beta}(v),x^{-1}]=[X_{\beta}(v),h^{-1}]\cdot[X_{\beta}(v),a^{-1}]^h.
$$
By the generalized Chevalley commutator formula we have
$[X_{\beta}(v),a^{-1}]^h\in U_{\{k\beta\mid k\ge 2\}}(R)$, and hence
$[\X_\beta(R),h^{-1}]\le U_{\{k\beta\mid k\ge 2\}}(R)$. By Corollary~\ref{cor:levi-cent-beta} this implies
that $[h,U_{(\beta)}(R)]=1$. Then $[a,\X_\beta(R)]=1$.
By Lemma~\ref{lem:centr-beta} this implies that $a\in\X_{m\beta}(R)$.
\end{proof}

The following easy fact will be used twice in the proof of Lemma~\ref{lem:InPQ}.

\begin{lem}\label{lem:omega}
Let $P$ be a parabolic subgroup of $G$
and $g=ahu\in\Omega_P(R)$, where $a\in U_P(R)$, $h\in L_P(R)$, and $u\in U_P^-(R)$.
If an element $x\in L_P(R)$ commutes with $g$, then it commutes with $a$, $h$, and $u$.
\end{lem}

\begin{proof}
By the condition we have $1=[x,ahu]^a=[a^{-1},x]\cdot[x,hu]$.
Hence $[x,a^{-1}]=[x,hu]\in U_P(R)\cap P^-(R)=\{1\}$. Therefore,
$x$ commutes with $a$ and $[x,hu]=[x,h]\cdot\conj{h}[x,u]=1$. Now, $[h,x]^h=[x,u]\in L_P(R)\cap U_P^-(R)=\{1\}$.
Thus, $x$ commutes with $u$ and $h$ as well.
\end{proof}

\begin{lem}\label{lem:InPQ}
Suppose that the absolute root system of $G$ is irreducible and its structure constants
are invertible in $R$. Let $Q$ be a parabolic $R$-subgroup of $G$ with a Levi subgroup $L_Q$
such that $\opn{rank}(\Phi_Q)\ge 2$.
For a proper parabolic subgroup $P\ge Q$ choose $L_P\ge L_Q$. Then, given $g\in\Omega_P(R)\sm\Center\bigl(G(R)\bigr)$
there exists a proper parabolic subgroup $M\ge L_Q$ of $G$ such that
$[g^{E(R)}\cap M(R),E(R)]\neq 1$.
%
%and let $L_Q$ be its Levi subgroup such that the root subschemes corresponding to $(Q,L_Q)$ are
%correctly defined over $R$ and
%\begin{aside}
%I do not understand the previous sentence. If there exists $P\lneqq Q$ and $\Phi_Q$ is connected, then
%$\opn{rank}(\Phi_Q)\ge 2$, isn't it?
%\end{aside}
%Let $P\ge Q$ be another
%parabolic subgroup with a Levi subgroup $L_P$ such that $L_Q\lneqq L_P$.
%Then for any $g\in U_Q(R)P^-(R)$ satisfying
%$[g,E(R)]\neq 1$ there exists a parabolic subgroup $M\ge L_Q$ of $G$ such that
%$[g^{E(R)}\cap M(R),E(R)]\neq 1$.
\end{lem}

\begin{proof}
By Theorem~\ref{thm:E-cent} $\Center\bigl(G(R)\bigr)=\Center(G)(R)$ is contained in every parabolic subgroup, hence
by Lemma~\ref{lem:HallWitt} it is enough to prove the claim under the assumption that $G$ is adjoint.
Let $\beta\in\Phi^Q$ be a simple relative root. Denote by $Q_{(\beta)}$ the smallest parabolic subgroup
of $G$ containing $Q$ and $\X_{-\beta}$. Note that since $\opn{rank}(\Phi_Q)\ge 2$, the parabolic subgroup $Q_{(\beta)}$
is proper. In case $P=Q$ we substitute $P$ by $Q_{(\beta)}$ using the fact that
$\Omega_Q\subseteq\Omega_{Q_{(\beta)}}$. Thus, we may assume that $L_P$ contains $\X_\beta$ for some simple
relative root $\beta\in\Phi_Q$.

Write $g=ahu$, where $a\in U_P(R)\le U_Q(R)$ and $h\in L_P(R)$, and $u\in U_P^-(R)$.
%%Choose any total order on $\Phi_Q^+$ compatible with the height of roots.
Denote by $\opn{lev}_P(g)$ the largest vector $\alpha\in\Phi_Q^+\cup\{0\}$ such that $a\in U_Q^{>\alpha}(R)$.
We proceed by going down induction on $\opn{lev}_P(g)$. If $\opn{lev}_P(g)$ is the maximal relative root,
then $a=1$ and $g$ already lies in the parabolic subgroup $P^-$. Otherwise,
consider the element
$$
g'=[X_\beta(v),g]^a=[a^{-1},X_\beta(v)]\cdot [X_\beta(v),b]\in g^{E(R)}.
$$
Here $[X_\beta(v),b]\in P^-(R)$, while $[a^{-1},X_\beta(v)]\in U_Q^{>\alpha+\beta}(R)$.
Thus, $g'\in\Omega_P(R)$ and $\opn{lev}_P(g')>\opn{lev}_P(g)$. Hence, if $g'$ does not belong to
the center at least for one simple root $\beta\in\Phi_Q$ such that $\X_\beta\le L_P$ and
$v\in V_\beta$, we are done by the induction hypothesis.

Now, assume that $g'$ lies in the center of $G(R)$ for all $\beta$ as above and all $v\in V_\beta$. Since $G$ is adjoint,
we have $g'=1$, i.\,e. $\X_\beta(R)$ commutes with $g$. By Lemma~\ref{lem:omega}
$\X_\beta(R)$ commutes with $a$, $h$, and $u$ as well.
Denote by $Q'$ the parabloic subgroup $Q\cap L_P$ of the reductive group $L_P$ and
let $\Delta_{Q'}$ be the root system of $L_P$ relative to $Q'$.
Note that $h\in L_P$ commutes with $\X_\beta(R)$ for all simple roots $\beta\in\Delta_{Q'}$.
Therefore, for any maximal ideal $\m$ of $R$ the image of $h$ in $L_P(R/\m)$
commutes with $\X_\beta(R/\m)$.
Since the structure constants of the absolute root system of $G$ are invertible in $R$,
the same is true for the absolute root system of $L_P$.
By Lemma~\ref{lem:const} the groups $\X_\beta(R/\m)$ span $U_{Q'}(R/\m)$ as $\beta$ ranges
over all simple roots of $\Delta_{Q'}$. Thus, the image of $h$ in $L_P(R/\m)$ commutes with
$U_{Q'}(R/\m)$, and, hence, belongs to $Q'(R/\m)\subseteq\Omega_{Q'}(R/\m)$ by Lemma~\ref{lem:u-cent-field}.
%%It follows that the image of $g'$ in $G(R/\m)$ belongs to $\Omega_{Q'}(R/\m)$.
Since $\Omega_{Q'}$ is an open subscheme of $L_P$,
by Lemma~\ref{lem:open-fields} we have $h\in\Omega_{Q'}(R)$.
Then
$$
g=ahu\in U_P(R)\cdot \Omega_{Q'}(R)\cdot U_{P^-}(R)\subseteq \Omega_Q(R).
$$

Let $P'$ be the minimal parabolic subgroups of $G$, containing $Q$ and $\X_{-\beta}$.
Since $g\in\Omega_Q(R)\subseteq\Omega_{P'}(R)$, we can write $g=a'h'u'$ for some
$a'\in U_{P'}$, $h'\in L_{P'}$, and $u'\in U^-_{P'}$. Recall that $g$ commutes with
$\X_\beta(R)$. Then, Lemma~\ref{lem:omega} asserts that $a'$, $h'$, and $u'$ commute with
$\X_\beta(R)$ as well.
By Lemmas~\ref{lem:small-levi-b}
and~\ref{lem:centr-beta} this implies that $g$ is a product of root elements of the form $X_\alpha(u_\alpha)$ where
all $\alpha$ satisfy $\alpha+\beta\notin\Phi_Q$, and of an element of $L_Q(R)$. Then by Lemma~\ref{lem:parab-centr-root}
there is a proper parabolic set of relative roots $\Sigma\subset\Phi_Q$ such that $g\in U_\Sigma(R)$ in the notation
of Lemma~\ref{lem:T-P}. Then we take $M=U_\Sigma$.
\end{proof}

The following statements follow immediately from Lemmas~\ref{lem:InPQ} and~\ref{lem:InP}.

\begin{cor}\label{cor:InPQ}
Under assumptions of Lemma~$\ref{lem:InPQ}$ the subgroup $g^{E(R)}$ contains a nontrivial root unipotent element.
\end{cor}

\begin{cor}\label{cor:UnderRad}
Suppose that $G$ is semisimple, $R$ is connected, the absolute root system $\Phi$ of $G$ is irreducible,
its structure constants are invertible in $R$, and $G$ contains $(\mathbb{G}_{m,R})^2$.
If $H\cap G(R,\opn{Rad}R)\not\subseteq\Center\bigl(G(R)\bigr)$,
then $H$ contains a nontrivial root unipotent element.
\end{cor}

\begin{proof}
Since $G$ contains $(\mathbb{G}_{m,R})^2$, there is a parabolic subgroup $Q$ of $G$ such that
$\opn{rank}(\Phi_Q)\ge 2$. Take $g\in (H\cap G(R,\opn{Rad}R))\setminus\Center\bigl(G(R)\bigr)$.
Then $\rho_\m(g)\in\Omega_Q(R/\m)$
for any maximal ideal $\m$ of $R$. Then by Lemma~\ref{lem:open-fields} $g\in\Omega_Q(R)$,
and the result follows from Corollary~\ref{cor:InPQ}.
\end{proof}

%%==============================================================

\section{Proof of the normal structure theorem}\label{NormalSec}

The next lemma is a key step to the proof of the normal structure theorem.

\begin{lem}\label{lem:generic}
Let $G$ be a reductive group scheme over a Noetherian connected ring $K$, and let $Q$ be a parabolic $K$-subgroup of $G$ such that
$\Phi_Q$ is irreducible and $\opn{rank}(\Phi_Q)\ge 2$.
For any parabolic subgroup $Q\lneqq P$
there exist elements $c_{ij}\in g^{E(A)}\cap\Omega_P(A)$, $i=1,\dots,l$, $j=1,\dots,n$,
satisfying the following property.
Let $S$ be a subscheme defined by the formula
$$
S(R)=\{h\in G(R)\mid h(c_{ij})\in\Center(G)(R)\text{ for all }i=1,\dots,l\text{ and }j=1,\dots,n\}
$$
for any $K$-algebra $R$.
Then the group of points $S(F)$ does not contain $E(F)$ for any $K$-algebra $F$
that is a field.
\end{lem}

\begin{proof}
Since the subscheme $\Omega_P$ is open, hence it is covered by principal open subschemes
$\opn{Sp}_K A_{s_i}$, $i=1,\dots,l$. Let $\alpha,\beta\in\Phi_Q^-$ be two relative roots such that $-\alpha,-\beta$ are simple roots,
$\alpha+\beta\in\Phi_Q$, $\X_{\alpha}\subseteq L_P$, $\X_\beta\subseteq U_{P^-}$.
Let $\{e_1,\dots,e_n\}$ be a set of generators of the $K$-module $V_\alpha$.
Choose $i=1,\dots,l$ and $j=1,\dots,n$ and put $s=s_i$ and $v=e_j$.

Let $g_s$ be the image of $g$ in $G(A_s)$. In other words, $g_s=\lambda_s$ is the localization
homomorphism. It is a tautology that $g_s$ factors through $A_s$, which means that
$g_s\in\opn{Sp}_KA_s(A_s)\subseteq\Omega_P(A_s)$.
Decompose it into a product $g_s=ab$, where  $a\in U_P(A_s)$ and $b\in P^-(A_s)$.
Since $A$ is Noetherian, by Lemma~\ref{lem:Bak} there exists $m\in\N$ such that the restriction of $\lambda_s$ to
$s^mA$ is injective.
By Lemma~\ref{lem:ClearDenom}
there exists a positive integer $k\ge m$ such that
$$
d_s=\conj{a}{X_\alpha(s^kv)}\in\lambda_s\bigl(U_P(A,s^mA)\bigr)
\text{ and }
f_s=d_s^{g_s}=X_\alpha(s^kv)^{b}\in\lambda_s\bigl(P^-(A,s^mA)\bigr).
$$
Let $d=\lambda_s^{-1}(d_s)\in U_P(A,s^mA)$ and $f=\lambda_s^{-1}(f_s)\in P^-(A,s^mA)$
(by definition of $m$ these preimages are unique).
Put $c_{ij}=c=[g^{-1},d]\in G(A,s^mA)$. Then
$$
\lambda_s(c)=[g_s^{-1},d_s]=d_s^{g_s}d^{-1}=f_sd_s^{-1}\in\lambda_s\bigl(\Omega_P(A)\bigr).
$$
By definition of $m$ we have $c_{ij}=fd^{-1}\in\Omega_P(A)$.

Let $F$ be a field.
Put $h=X_\beta(u)\in\Omega_P(F)=\bigcup\limits_{i=1}^l\opn{Sp}_KA_{s_i}(F)$
for some $u\in V_\beta\otimes_KF\sm\{0\}$. Choose $i$ such that $h\in\opn{Sp}_KA_{s_i}$ and put
$s=s_i$. Then $h$ factors through $A_s$, i.\,e. $h=\wt h\circ\lambda_s$ for some
$\wt h:A_s\to F$. Since $h(g)=h$, we have
$
h=\wt h(g_s)=\wt h(a)\wt h(b)=e\cdot X_\beta(u).
$
The uniqueness of representation of $h$ as a product
of an element from $U_P(F)$ by an element from $P^-(F)$ implies that $h(a)=e$.
Thus, we get
$$
h(c_{ij})=[h(g^{-1}),h(d)]=[X_\beta(u)^{-1},X_\alpha(s^ke_j)].
$$
By Lemma~\ref{lem:ABe} there exists $j$ such that $N_{\alpha,\beta,1,1}(e_j,u)\ne 0$.
By the generalized Chevalley commutator formula $h(c_{ij})\notin\Center(G)(F)$. Thus,
$h\notin S(F)$, and hence $S(F)$ does not contain $E(F)$, as required.
\end{proof}

Now we are ready to prove the main theorem of the present paper.

\begin{proof}[Proof of Theorem~$\ref{thm:main}$]
Let $T\le G$ be the subgroup isomorphic to $(\mathbb{G}_{m,K})^2$.
By Lemma~\ref{lem:T-P} there are two parabolic subgroups $Q=Q^+$ and $Q^-$ of $G$
with a common Levi subgroup $L_Q=\Center_G(T)$.

Assume first that $K$ is Noetherian. Then $K$ is a finite product of connected Noetherian rings, so we can assume without loss
of generality that $K$ is Noetherian and connected.  The relative root system $\Phi_Q$ of Lemma~\ref{lem:relroots}
is irreducible and $\opn{rank}(\Phi_Q)\ge 2$. By Proposition~\ref{prop:stand-comm-formula} we have $E(K,\q)=[G(K,\q),E(K)]$.
%%By Lemma~\ref{lem:E_P} we have $E_P(K,\q)=E_Q(K,\q)$ for any other parabolic subgroup $P$.
Thus,
by Proposition~\ref{lem:reduction} it suffices to prove that if $H$ is not inside $\Center\bigl(G(K)\bigr)$, then
it contains a nontrivial relative root unipotent element.

Let $\beta\in\Phi_Q$ be a simple root. Then by Lemma~\ref{lem:T-P}
$P=P^+=U_{\Phi_Q^+\cup\Z\beta}$ is a parabolic subgroup distinct from
$Q$, with a Levi subgroup $L_P=U_{\Z\beta}$. Let
$S$ be a subscheme of $G$ satisfying conditions of Lemma~\ref{lem:generic}.
If there exists $h\in H$ such that $h(c_{ij})\notin\Center\bigl(G(K)\bigr)$, then $h(c_{ij})\in h^{E(K)}\cap \Omega_P(K)$
is subject to Corollary~\ref{cor:InPQ}, and hence $H$ contains a nontrivial
relative root unipotent element.

Otherwise $H$ is contained in the set of $K$-points of the subscheme $S$.
Consider a maximal ideal $\m$ of $K$. The image $\ovl H$ of the subgroup $H$ under the canonical
homomorphism $G(K)\to G(K/\m)$ is contained in $S(K/\m)$ and is normalized by
$E(K/\m)$. Tits' simplicity theorem~\cite{Tits64} then implies that either $H$ contains $E(K/\m)$, or $H$ is contained in
$\Center\bigl(G(K/\m)\bigr)$. Since $E(K/\m)$ is not contained in $S(K/\m)$,
$\ovl H$ is contained in $\Center\bigl(G(K/\m)\bigr)$.
It follows that the image of the subgroup
$[H,E(K)]$ vanishes in $G(K/\m)$, i.\,e. $[H,E(K)]\le G(K,\m)$.

Since $\m$ is an arbitrary maximal ideal, we have $[H,E(K)]\subseteq G(K,J)$, where $J$ is the Jacobson
radical of $K$. On the other hand, by Lemma~\ref{lem:HallWitt} $[H,E(K)]$ is not contained in the center of $G(K)$.
Hence, by Corollary~\ref{cor:UnderRad} this subgroup contains a nontrivial relative root unipotent element.
This completes the proof of the Noetherian case.

Now let $K$ be arbitrary. For any finite set of elements $\Lambda\subseteq H$ there is a finitely generated subring
$\wt K$ of $K$ and a semisimple reductive group scheme $\wt G$ over $\wt K$ with a subgroup $\wt T\cong (\mathbb{G}_{m,\wt K})^2$
such that $G=\wt G_K$, $T=\wt T_K$,
and $\Lambda\subseteq \wt G(\wt K)$. Clearly, by Lemma~\ref{lem:T-P} there is also a parabolic subgroup $\wt Q\le\wt G$
such that $Q=\wt Q_K$.

Since for any maximal ideal $\wt{\m}$ of $\wt K$ there is a maximal ideal
$\m$ of $K$ such that $\wt{\m}\subseteq \m$, we conclude that $\wt G(\wt K/\wt{\m})$ also has irreducible root system whose structure
constants are invertible in $\wt K$. Thus $\Lambda^{E(\wt K)}\le \wt G(\wt K)$ is subject to the Noetherian case of the theorem.
Hence there exists an ideal $\q(\Lambda)\subseteq \wt K$ such that
$$
E\bigl(\wt K,\q(\Lambda)\bigr)\le \Lambda^{E(\wt K)}\le C\bigl(\wt K,\q(\Lambda)\bigr).
$$
Clearly, $\q(\Lambda)$ is uniquely determined by $\wt K$ and $\Lambda$.
Let $\q$ be the ideal of $K$ generated by all subsets $\q(\Lambda)\subseteq\wt K\subseteq K$. Then, clearly, $H\le C(K,\q)$.
W show that $E_Q(K,\q)\le H$. In order to do that,
it is enough to check that
$U_{Q^\pm}((a))$ is contained in $H$ for any finite $K$-linear combination $a=\sum c_ia_i$ of elements $a_i\in\q(\Lambda_i)$.
Let $\wt K$ be subring of $K$ corresponding to the finite set
$\Lambda=\cup\Lambda_i$, and let $K'$ be the subring of $K$ generated by $\wt K$ and all $c_i$. Then by the Noetherian
case of the theorem applied to $\wt G_{K'}$, we conclude that
$U_{\wt Q^\pm}((a))=U_{Q^\pm}((a))$ are contained in $\Lambda^{E(K')}\le \Lambda^{E(K)}\le H$.

The same argument as above also shows that $E_Q(K,\q)=[G(K,\q),E(K)]$ and $E_P(K,\q)=E_Q(K,\q)$ for any other parabolic subgroup $P$ of $G$, since this
equality holds for each finitely generated subring $\wt K$ such that $P$ is defined over $\wt K$ and the ideal $\wt\q=\wt K\cap \q$.
\end{proof}

%%==============================================================
%\bibliographystyle{zapiski}

\bibliographystyle{amsalpha}
\bibliography{english}
%%==============================================================
\end{document}